\documentclass[12pt,a4paper]{article} 
\usepackage[utf8]{inputenc}
\usepackage[T1]{fontenc}
\usepackage{authblk}
\usepackage{amsmath}
\usepackage{amssymb}
\usepackage{amsthm,mathtools}
\usepackage{lineno}
\usepackage[mathscr]{eucal}
\usepackage{xcolor}
\usepackage{fontenc}
\usepackage{graphicx}
\usepackage{geometry}
\usepackage{lineno}
\theoremstyle{definition}
\newtheorem{definition}{Definition}[section]
\newtheorem{theorem}[definition]{Theorem}
\newtheorem*{theorem*}{Conjecture}
\newtheorem{proposition}[definition]{Proposition}
\newtheorem{lemma}[definition]{Lemma}
\newtheorem{corollary}[definition]{Corollary}
\theoremstyle{remark}
\newtheorem{remark}[definition]{Remark}
\newtheorem{example}[definition]{Example}
\newcounter{enumctr}

\newcommand{\R}{\mathbb{R}}
\newcommand{\N}{\mathbb{N}}

\newcommand{\C}{\mathbb{C}}

\providecommand{\keywords}[1]{\textbf{\textbf{Key words: }} #1}
\begin{document}
\title{Qualitative analysis of solutions to mixed-order positive linear coupled systems with  bounded or unbounded delays}         



\author[1,2]{Hoang The Tuan\thanks{Corresponding Author: httuan@math.ac.vn}}
\author[3]{La Van Thinh\thanks{lavanthinh@hvtc.edu.vn}}
\affil[1]{Institute of Mathematics, Vietnam Academy of Science and Technology, 18 Hoang Quoc Viet, Hanoi 10307, Viet Nam}
\affil[2]{Department of Mathematics, Great Bay University, Dongguan, Guangdong 523000, China}
\affil[3]{Academy of Finance, No. 58, Le Van Hien St., Duc Thang Wrd., Bac Tu Liem Dist., Hanoi, Viet Nam}


\maketitle
\begin{abstract}
This paper addresses the qualitative theory of mixed-order positive linear coupled systems with bounded or unbounded delays. First, we introduce a general result on the existence and uniqueness of solutions to mixed-order linear systems with time-varying delays. Next, we obtain necessary and sufficient criteria which characterize the positivity of mixed-order delay linear coupled systems. Our main contributions are in Section 5. More precisely, by using a smoothness property of solutions to fractional differential equations and developing a new appropriated comparison principle for solutions to mixed-order delay positive systems, we prove the attractivity of mixed-order non-homogeneous linear positive coupled systems under the impact of bounded or unbounded delays. We also establish a necessary and sufficient condition to ensure the stability of homogeneous systems. As a consequence of these results, we show the smallest asymptotic bound of solutions to mixed-order delay positive non-homogeneous linear coupled systems where disturbances are continuous and bounded. Finally, we provide numerical simulations to illustrate the proposed theoretical results.
\end{abstract}
\section{Introduction}
Positive Systems are dynamical systems possessing an important property that any nonnegative input and nonnegative initial state generate a nonnegative state trajectory and output at all times. These systems appear frequently in practical applications as lots of physical quantities (such as concentrations, population levels, queue lengths, charge levels, and light intensity levels) and chemical reaction networks in which the state variables represent reactant concentrations are naturally constrained to be positive. The unifying approach of the system theory which is nowadays known as “Positive System Theory” was initiated in the '80s by David Luenberger in his fundamental book \cite{Luenberger}. From that time on, an impressive number of theoretical and applicative contributions to this field has appeared, e.g., medicine \cite{Carson}, \cite{Vargas}, pharmacology \cite{Coxson}, biology, ecology and physiology \cite{Haddad}, biomolecular \cite{Vecchio}, \cite{Moreno}, traffic and congestion modelling \cite{Blanchini}, filtering and charge routing networks \cite{Benvenuti}, econometrics \cite{Nieuwenhuis}. For more details, interested readers see the tutorial paper by A. Rantzer and M.E. Valcher \cite{Rantzer} and the references therein.

In distributed systems with the exchange of information involved, delays are inevitable. They occur due to finite velocities of signal propagation or processing delays leading to memory effects. Hence, up to now, a considerable effort has been made to study the qualitative theory and performance of delay systems. A systematic analysis of delay differential equations has been carried out by J.K. Hale and S.M. Verduyn Lunel \cite{Hale}. In this famous monograph, the authors have discussed the existence, uniqueness of solutions, and Lyapunov-type methods for proving the stability of an equilibria of dynamic systems under different types of delay effects. In addition, the existence of periodic solutions, the existence of stable manifolds, and the Hopf bifurcation of systems are also investigated. The global research of nonlinear delay equations is emphasized in \cite{Kuang}. Specifically, the author is interested in the possible influence of delays on the dynamics to the system, including topics such as the stability switching for increasing delays, the long-time coexistence of populations, and oscillatory aspects of the dynamics. Using a unitary methodology as the eigenvalue-based approach, in \cite{Michiels}, W. Michiels and S.I. Niculescu have presented an stability analysis and synthesis by delayed (state and output) feedback in the linear case. Generally, they have mainly concentrated on the two aspects: sensitivity analysis concerning delays and other systems' parameters and pole placement strategies in stabilization and optimization of the spectral abscissa function or robustness indicators.

The influence of delays on the dynamics of linear positive systems and some classes of nonlinear monotone systems have been well studied by many authors. We only mention here some of the most significant of them. Using linear Lyapunov–Krasovskii functionals, in \cite{Haddad_04}, W.M. Haddad and V. Chellaboina have developed necessary and sufficient conditions for the asymptotic stability of nonnegative linear systems with a constant delay. With the help of the semigroup theory and an eigenvalue-based approach, P.H.A. Ngoc \cite{Ngoc_09} has obtained an explicit criterion for the positivity of the solution semigroup of linear Volterra integro-differential systems with infinitely many delays. By applying co-positive linear Lyapunov functions along with the dissipativity theory, in \cite{Brat_13}, C. Briat has analyzed the stability and control of uncertain linear positive systems. Exploiting a comparison principle, in \cite{ShenLam15}, J. Shen and J. Lam have proven that the $\ell_\infty$-gain of positive linear systems with unbounded delays is independent of the magnitude of delays. 

Coupled differential–difference equations (systems composed of an ordinary differential equation coupled with a continuous time difference equation) with delays have received extensive attention in engineering applications like circuit and fluid system, see e.g., \cite{Niculescu, Rasvan_02, Rasvan_06} and and the references therein. They cover several types of classic dynamic systems, such as singular systems, systems with multiple commensurate delay, and dynamic systems with internal constraints, see e.g., \cite{Gu_06, Gu_10, Mazenc, Huang_21} and the references therein. By combining the comparison arguments and the entry-wise monotonicity of positive coupled differential-difference equations, in \cite{Shen_15}, J. Shen and W. X. Zheng have investigated the asymptotic stability of internally positive coupled differential-difference equations with time-varying delays. After that, the stability of positive coupled differential-difference systems with unbounded delays is shown in P.N. Pathirana et al. \cite{Hieu_18} and Q. Xiao et al. \cite{Huang_21} by developing new comparison techniques for solutions. Finally, we introduce an outstanding contribution published by H.R. Feyzmahdavian et al. \cite{Feyzmahdavian}. In that paper, by constructing special Lyapunov candidate functions, ingenious solution comparison techniques, and appropriated weighted norms, the authors have established the delay-independent stability of a class of nonlinear positive systems, which includes positive linear systems as a special case, and allows delays are possibly unbounded.

While, as reviewed above, the asymptotic behavior of solutions to delay positive systems is surveyed thoroughly, the theory for fractional-order positive systems is considerably less well-developed. In our opinion, in this direction, one of the most interesting works is the paper by J. Shen and J. Lam \cite{ShenLam_16} in which the authors have shown the stability and performance analysis of mixed-order positive linear systems with bounded delays. Then, in \cite{TTL_21}, H.T. Tuan, H. Trinh, and J. Lam obtained a criterion to classify the stability of mixed-order linear systems where time-varying delays are bounded or unbounded. The approach in the two papers is the comparison principle of positive systems with various time-varying delays. Based on delayed Mittag–Leffler type matrix functions, in \cite{Mahmudov}, I.T. Huseynov and N.I. Mahmudov have investigated the existence, uniqueness and the Ulam–Hyers stability for fractional-order neutral differential equations with a constant delay. In light of vector Lyapunov type functions and comparison arguments, K.L. Bichitra, and N.B. Swaroop \cite{BS_22}, J.A. Gallegos et al. \cite{GCM_20} have described the asymptotic behavior of solutions to some multi-order nonlinear systems under the influence of multiple time-varying delays, J. Jia et al. \cite{Jia} have considered the global stabilization of fractional-order memristor-based neural networks with incommensurate orders and multiple time-varying delays. Recently, based on the  the separation of solutions of scalar fractional differential equations, R.A.C. Ferreira \cite{Ferreira} has shown the positivity of solutions of scalar linear fractional differential equations and used this result to deduce necessary
optimality conditions for a fractional variational problem of Herglotz type.

In the present paper, we focus on mixed-order linear coupled systems with bounded or unbounded delays
\begin{equation} \label{E0}
\begin{cases}
^{\!C}D^{\hat{\alpha}}_{0+}x(t)&=Ax(t)+Bx(t-\tau_1(t))+Ey(t-\tau_2(t))+f(t),\; t\in (0,\infty),\\
y(t)&=Cx(t)+Dy(t-\tau_3(t))+g(t),\; t\in [0,\infty),\\
x(\cdot)&=\psi(\cdot),\;y(\cdot)=\varphi(\cdot)\;\text{on}\; [-r,0],
\end{cases}
\end{equation}
here $\hat\alpha = (\alpha_1,\dots,\alpha_d) \in (0,1]\times\dots\times (0,1]$, $ ^C D^{\hat\alpha}_{0^+} x(t) := \left(^C D^{\alpha_1}_{0^+}x_1(t),\dots,^C D^{\alpha_d}_{0^+}x_d(t) \right)^{\rm T}$ ($^C D^{\alpha_i}_{0^+}x_i(t)$ is the Caputo fractional derivative of the order $\alpha_i$ which is defined later), $A,B\in \R^{d\times d}$, $E\in \R^{d\times n}$, $C\in \R^{n\times d}$, $D\in \R^{n\times n}$, $f:[0,\infty)\rightarrow \R^d$, $g:[0,\infty)\rightarrow \R^n$, $\tau_k:[0,\infty)\rightarrow \R_{\geq 0}$ $(k=1,2,3)$, $\psi:[-r,0]\rightarrow \R^d,$ and $\varphi:[-r,0]\rightarrow \R^n$ are continuous. It is a continuation of our recent researches on the qualitative theory of fractional-order dynamical systems, see, e.g., \cite{TuanC_20, Tuan_SIAM, TTL_21}. Our main contribution in this paper (Section \ref{s5}) is to propose a criterion on the asymptotic behavior of the system \eqref{E0} in the case its positivity is satisfied. Due to the complexity of the research topic (the system contains many different fractional order derivatives and non-locality of solutions), we especially note that it is not clear how to directly apply the techniques used to analyze positive systems (with or without delay) as eigenvalue-based methods, Lyapunov-type candidate function methods for the considered problem. To achieve the desired goal, we have developed a new approach based on a specific feature of fractional-order differentiable functions and clever solution comparison techniques suitable for fractional-order positive systems.

We organize the article as follows. In the section \ref{s2}, we list some preparation results needed for further analysis in the subsequent sections. The section \ref{s3} deals with the existence and uniqueness of solutions to mixed-order coupled systems with time-varying delays. Positivity of the system \eqref{E0} is discussed in the section \ref{s4}. Asymptotic behavior of solutions to mixed-order positive linear coupled systems when delays are bounded or unbounded is indicated in the section \ref{s5}. At the end of this section, we provide some numerical examples to illustrate the effectiveness of the proposed main theoretical results.

To conclude the introduction, we present some notations that will be used throughout the rest of the paper. $\N$ is the set of natural numbers, $\R$, $\R_{\geq 0}$, $\R_{+}$ is the set of real numbers, nonnegative real numbers, and positive real numbers, respectively. $\R^d$ is the $d$-dimensional Euclidean space endowed with a norm $\|\cdot\|$, $\R^d_{\geq 0}:=\{x=(x_1,\dots,x_d)^{\rm T}\in\R^d:x_i\geq 0,\;i=1,\dots,d\}$, $\R^d_{+}:=\{x=(x_1,\dots,x_d)^{\rm T}\in\R^d:x_i> 0,\;i=1,\dots,d\}$ and ${\bf 1}_d\in \R^d$ is a vector with entries equal to 1. For $m,n\in\N$, $\R^{m\times n}$ is the set of $m\times n$ real matrices. Assume that $X$ is a Banach space equipped a norm $\|\cdot\|_X$, set $B_X(0;r):=\{x\in X:\|x\|_X\leq r\}$. For any $[a,b]\subset \R$, denote $C([a,b];X)$ the space of continuous functions $\xi:[a,b]\rightarrow X$ with the sup norm $\|\cdot\|_X$, i.e.,
\[
\|\xi\|_X:=\sup_{a\le t\le b}\|\xi(t)\|,\quad\forall \xi\in C([a,b];X).
\]
Let a matrix $A=(a_{ij})_{1\leq i,j\leq d}\in \R^{d\times d}$. We denote the spectrum of $A$ by $\sigma{(A)}$, that is $\sigma(A):=\{\lambda\in\mathbb C: {\text{det}}(A-\lambda I_d)=0\}$, where $I_d$ is the $d\times d$ identity matrix. The matrix $A$ is called Metzler if $a_{ij}\geq 0$ for all $1\leq i\ne j\leq d$. It is said to be Hurwitz if $\sigma{(A)}$ satisfies $$\sigma{(A)}\subset\{\lambda\in\C:\Re{\lambda}<0\}.$$
The inverse of $A$ is $A^{-1}$ only when $AA^{-1} = A^{-1}A = I_d$. Define $\rho(A):=\max\{|\lambda|:\lambda\in \sigma(A)\}$. The matrix $A$ is Schur if $\rho(A)<1$. Let $n,m\in \N$ and $A=(a_{ij})_{1\leq i\leq n}^{1\leq j\leq m},B=(b_{ij})_{1\leq i\leq n}^{1\leq j\leq m}\in \R^{n\times m}$. We write $A\succ B$ ($A\succeq B$) if $a_{ij}>b_{ij}$ ($a_{ij}\geq b_{ij}$, respectively) for all $1\leq i\leq n,\;1\leq j\leq m$. The matrix $A$ is nonnegative if $a_{ij}\geq 0$ for all $1\leq i\leq n$, $1\leq j\leq m$. For $\alpha \in (0,1],$ we define the Riemann-Liouville fractional integral of a function $x :[0,T] \rightarrow \mathbb R$ as 
$$ I^\alpha_{0^+}x(t) := \frac{1}{\Gamma(\alpha)}\int_{0}^{t}(t-s)^{\alpha -1}x(s)ds,\;t\in [0,T],$$
and its Caputo fractional derivative of the order $\alpha\in (0,1]$ as 
$$ ^C D^\alpha_{0^+}x(t) := \frac{d}{dt}I^{1-\alpha}_{0^+}(x(t)-x(0)), \;t\in (0,T],$$
where $\Gamma(\cdot)$ is the Gamma function and $\frac{d}{dt}$ is the usual derivative.
Let $\hat\alpha = (\alpha_1,\dots,\alpha_d) \in (0,1]\times\dots\times (0,1]$ be a multi-index and $x = (x_1,\dots, x_d)^{\rm T}$ with $x_i :[0,T] \rightarrow \mathbb R$, $i=1,\dots,d,$ is a vector valued function, then
$$ ^C D^{\hat\alpha}_{0^+} x(t) := \left(^C D^{\alpha_1}_{0^+}x_1(t),\dots,^C D^{\alpha_d}_{0^+}x_d(t) \right)^{\rm T}.$$
On the Caputo fractional derivative, see e.g., \cite[Chapter III]{Kai} and \cite{Vainikko_16} for more details.
\section{Preliminary}\label{s2}
We introduce some preparation results needed for further analysis in the next parts.
\begin{lemma}
\label{L1} Let $A\in \R^{d\times d}$ and suppose that it is Metzler. Then, the matrix $A$ is Hurwitz if and only if one of the following two conditions is verified.
\begin{itemize}
\item[(i)] There exists a vector $\lambda \succ 0$ such that $A\lambda \prec 0.$
\item[(ii)] The matrix $A$ is invertible and $-A^{-1}\succeq 0.$
\end{itemize}
\end{lemma}
\begin{proof}
See, e.g., \cite[Proposition 2]{Rantzer}
\end{proof}
\begin{lemma}
\label{L2} Let $M\in \R^{d\times d}$ be nonnegative. Then, the following statements are equivalent.
\begin{itemize}
\item[(i)] $M$ is a Schur matrix.
\item[(ii)] There is a vector $\eta \succ 0$ such that $(M-I_d)\eta \prec 0.$
\item[(iii)] The matrix $I_d-M$ is invertible and $(I_d-M)^{-1}\succeq 0.$
\end{itemize}
\end{lemma}
\begin{proof}
See, e.g., \cite[Proposition 1]{Rantzer}.
\end{proof}
\begin{lemma} \label{L3} Let $\alpha \in (0,1]$ and $c \in \R$. Then, for all $t>0$, we have $$\frac{d}{dt}E_\alpha(c t^\alpha)=c t^{\alpha-1}E_{\alpha,\alpha}(c t^\alpha),$$
where $E_{\alpha,\beta}:\R\rightarrow \R$ is the classical Mittag-Leffler function defined by $E_{\alpha,\beta} (x):=\sum_{k=0}^\infty \frac{x^k}{\Gamma(\alpha k+\beta)}$ for all $x\in\R$, and $E_\alpha(\cdot):=E_{\alpha,1}(\cdot)$.
\end{lemma}
\begin{proof}
The proof is obvious, and thus we omit it.
\end{proof}
\begin{lemma} \label{L4} Let $\alpha \in (0,1]$ and $\lambda \in \R$. Suppose that $x(\cdot)\in C([0,\infty);\R)\cap C^1((0,\infty);\R)$ and the limit $\lim_{t\to 0}t^{\beta}\dot{x}(t)$ exists for some $\beta\in [0,1)$. Then,
$$\frac{d}{dt}\int_0^t (t-s)^{\alpha-1}E_{\alpha,\alpha}(\lambda(t-s)^{\alpha})x(s)ds=t^{\alpha-1}E_{\alpha,\alpha}(\lambda t^{\alpha})x(0)+\int_0^t (t-s)^{\alpha-1}E_{\alpha,\alpha}(\lambda(t-s)^{\alpha})\dot x(s)ds$$
for any $t>0$.
\end{lemma}
\begin{proof} By using Lemma \ref{L3} and the formula for integration by parts, for $t>0$, we obtain
\begin{align*}
y(t):&=\int_0^t (t-s)^{\alpha-1}E_{\alpha,\alpha}(\lambda(t-s)^{\alpha})x(s)ds
=-\frac{1}{\lambda}\int_0^tx(s)d[E_{\alpha}(\lambda(t-s)^{\alpha})]\\
&=-\frac{1}{\lambda}[x(s)E_{\alpha}(\lambda(t-s)^{\alpha})]|_0^t+\frac{1}{\lambda}\int_0^tE_{\alpha}(\lambda(t-s)^{\alpha})\dot x(s)ds\\
&=\frac{-x(t)+x(0)E_{\alpha}(\lambda t^\alpha)}{\lambda} + \frac{1}{\lambda}\int_0^tE_{\alpha}(\lambda(t-s)^{\alpha})\dot x(s)ds.
\end{align*}
Based on the well-known rule for the differentiation of an integral depending on a parameter with the upper limit depending on the same parameter (see, e.g., \cite[Formula (2.210), p. 98]{Podlubny}), then
\begin{align*}
\dot y(t)&=-\frac{\dot x(t)}{\lambda} + t^{\alpha-1}E_{\alpha,\alpha}(\lambda t^{\alpha})x(0)
+\frac{1}{\lambda}\int_0^t\frac{\partial E_{\alpha}(\lambda(t-s)^{\alpha})}{\partial t}\dot x(s)ds \\
&\hspace{2cm}+\frac{1}{\lambda} \lim_{s\to t^-}E_{\alpha}(\lambda(t-s)^{\alpha})\dot x(s)\\
&=t^{\alpha-1}E_{\alpha,\alpha}(\lambda t^{\alpha})x(0)+\int_0^t (t-s)^{\alpha-1}E_{\alpha,\alpha}(\lambda(t-s)^{\alpha})\dot x(s)ds,\;\forall t>0.
\end{align*}
The proof is complete.
\end{proof}
\section{Existence and uniqueness of solutions to mixed-order coupled systems with time-varying delays}\label{s3}
Let $\hat{\alpha}=(\alpha_1,...,\alpha_d) \in (0,1]\times...\times(0,1] \subset \R^d$, $A,B\in \R^{d\times d}$, $E\in \R^{d\times n}$, $C\in \R^{n\times d}$, $D\in \R^{n\times n}$, $f:[0,T]\rightarrow \R^d$, $g:[0,T]\rightarrow \R^n$ are continuous functions and $\tau_k:[0,T]\rightarrow \R_{\geq 0}$ $(k=1,2,3)$ is continuous and satisfies the following condition
\begin{itemize}
\item[(T1)] $t-\tau_k(t)\ge-r,\quad \forall t\in [0,T]$,
\end{itemize}
where $n,d\in \N$ and $r,T>0$ are arbitrary but fixed. In this section, we consider the following system
\begin{equation} \label{E1}
\begin{cases}
^{\!C}D^{\hat{\alpha}}_{0+}x(t)&=Ax(t)+Bx(t-\tau_1(t))+Ey(t-\tau_2(t))+f(t),\; t\in (0,T]\\
y(t)&=Cx(t)+Dy(t-\tau_3(t))+g(t),\; t\in [0,T]
\end{cases}
\end{equation}
with the initial conditions 
\begin{equation}\label{dkd}
\begin{cases}
x(\cdot)&=\psi(\cdot)\\
y(\cdot)&=\varphi(\cdot)
\end{cases}
\end{equation} 
on $[-r,0]$ in which $\psi:[-r,0]\rightarrow \R^d,\varphi:[-r,0]\rightarrow \R^n$ are continuous such that the compatibility condition is subjected

$\textup{(K)}\; C\psi(0)+D\varphi(-\tau_3(0))+g(0)=\varphi(0)$.
\begin{definition}
	A continuous vector valued function $\Phi(\cdot;\psi,\varphi):=(x(\cdot;\psi,\varphi), y(\cdot;\psi,\varphi))^{\rm T}: [-r,T]\rightarrow \R^{d+n}$ is called a solution to the system \eqref{E1}-\eqref{dkd} if it statisfies the equation \eqref{E1} and the initial condition \eqref{dkd} on $[-r,0]$.
	\end{definition}
Our first contribution in this paper is showing an existence and uniqueness theorem for the initial value problem \eqref{E1}-\eqref{dkd}.
\begin{theorem} \label{EU}
Suppose that $\|C\|_\infty,\|D\|_\infty<1$, that is,
	\begin{align}
		&\displaystyle\sum_{k=1}^d|c_{ik}|<1,\quad\forall i=\overline{1,n},\label{c1}\\
		&\displaystyle\sum_{j=1}^n|d_{ij}|<1,\quad \forall i=\overline{1,n}\label{d1}.
	\end{align} 
	Then, the system \eqref{E1}-\eqref{dkd} has a unique solution $\Phi(\cdot;\psi,\varphi)$ on $[-r,T]$, where $$\Phi(t;\psi,\varphi)=\left(\begin{array}{cc}
		x(t;\psi,\varphi) \\ y(t;\psi,\varphi)\end{array}\right) =\left(\begin{array}{cc}
		x_1(t;\psi,\varphi) \\ \vdots \\ x_d(t;\psi,\varphi) \\y_1(t;\psi,\varphi) \\ \vdots \\ y_n(t;\psi,\varphi)\end{array}\right) .$$
\end{theorem}
\begin{proof}
	First, it is easy to see that the initial value problem \eqref{E1}-\eqref{dkd} is equivalent to the following delay integral system
	\begin{align*}
		x_i(t)&=\psi_i(0)+\frac{1}{\Gamma(\alpha_i)}\int_0^t (t-s)^{\alpha_i-1}f_i(s)ds\\
		&+\frac{1}{\Gamma(\alpha_i)}\int_0^t (t-s)^{\alpha_i-1}\left(\sum_{j=1}^da_{ij}x_j(s)+\sum_{j=1}^db_{ij}x_j(s-\tau_1(s))+\sum_{k=1}^ne_{ik}y_k(s-\tau_2(s))\right)ds,\\
		y_j(t)&=\sum_{k=1}^dc_{jk}x_k(t)+\sum_{k=1}^nd_{jk}y_k(t-\tau_3(t))+g_j(t),\;t\in [0,T],
	\end{align*}
	for $i=1,...,d,\ j=1,...,n$, and $x(\cdot)=\psi(\cdot),\ y(\cdot)=\varphi(\cdot)$ on $[-r,0]$. We define on $C\left( [-r, T]; \mathbb R^{d+n} \right )$ a weighted norm as below
$$  ||\xi||_{\gamma} := \sup_{t \in [0, T]} \frac{\xi^*(t)}{e^{\gamma t}}, $$
where $\xi^*(t) = \displaystyle \max_{\theta \in [-r,t]} ||\xi(\theta)||$, $\|\cdot\|$ is the max norm on $\mathbb R^{d+n}$ and $\gamma$ is a large positive constant chosen later.
	
	Put $\Phi_0=(\psi,\varphi)^{\rm T}$ on $[-r,0]$ and $X=C\left( [-r, T]; \mathbb R^{d + n} \right )$.
We now establish a Lyapunov--Perron type operator $ \mathcal T_{\Phi_0} : X \rightarrow X$ associated with the system \eqref{E1}-\eqref{dkd} as follows. 

For any $\xi=(\xi^1,\xi^2)^{\rm T}\in \R^{d+n},$
	$$\mathcal{T}_{\Phi_0}\xi(t):=\left(\begin{array}{cc}
		(\mathcal{T}_{\Phi_0}^1\xi)(t) \\ (\mathcal{T}_{\Phi_0}^2\xi)(t)\end{array}\right),$$
	where
	\begin{align*}
		(\mathcal{T}_{\Phi_0}^1\xi )_i(t)&=\psi_i(0)+\frac{1}{\Gamma(\alpha_i)}\int_0^t (t-s)^{\alpha_i-1}\big(\sum_{j=1}^da_{ij}\xi^1_j(s)+\sum_{j=1}^db_{ij}\xi^1_j(s-\tau_1(s)))ds\\
		&\hspace{1,5cm}+\frac{1}{\Gamma(\alpha_i)}\int_0^t (t-s)^{\alpha_i-1}\big(\sum_{k=1}^ne_{ik}\xi^2_k(s-\tau_2(s))+f_i(s))ds,\\
		(\mathcal{T}_{\Phi_0}^2\xi)_j(t)&=\sum_{k=1}^dc_{jk}\xi^1_k(t)+\sum_{k=1}^nd_{jk}\xi^2_k(t-\tau_3(t))+g_j(t),
	\end{align*}
	for $t\in (0,T]$, $1\leq i\leq d,\ 1\leq j\leq n$ and $(\mathcal{T}_{\Phi_0}^1\xi)(\cdot)=\psi(\cdot),\;(\mathcal{T}_{\Phi_0}^2\xi )(\cdot)=\varphi(\cdot)$ on $[-r,0]$. 
	For any $\xi=(\xi^1,\xi^2)^{\rm T},\;\tilde{\xi}=({\tilde\xi}^1,{\tilde\xi}^2)^{\rm T}\in X$, $1\leq i\leq d$, we see on $[0,T]$ that
	\begin{align*}
		&\left|(\mathcal{T}_{\Phi_0}^1\xi)_i(t)-(\mathcal{T}_{\Phi_0}^1\tilde\xi)_i(t)\right|\\
		&\le \left|\frac{1}{\Gamma(\alpha_i)}\int_0^t (t-s)^{\alpha_i-1}\big(\sum_{j=1}^da_{ij}(\xi^1_j(s)-{\tilde\xi}^1_j(s))+\sum_{j=1}^db_{ij}(\xi^1_j(s-\tau_1(s))-{\tilde\xi}^1_j(s-\tau_1(s))))ds\right| \\
		&\hspace{1.5cm}+\frac{1}{\Gamma(\alpha_i)}\int_0^t (t-s)^{\alpha_i-1}\left|\sum_{k=1}^ne_{ik}(\xi^2_k(s-\tau_2(s))-{\tilde\xi}^2_k(s-\tau_2(s)))\right|ds\\
		&\leq\frac{1}{\Gamma(\alpha_i)}\int_0^t (t-s)^{\alpha_i-1}\left(\sum_{j=1}^d|a_{ij}||\xi^1_j(s)-{\tilde\xi}^1_j(s)|+\sum_{j=1}^d|b_{ij}||\xi^1_j(s-\tau_1(s))-{\tilde\xi}^1_j(s-\tau_1(s))|\right)ds\\
		&\hspace{1.5cm}+\frac{1}{\Gamma(\alpha_i)}\int_0^t (t-s)^{\alpha_i-1}\sum_{k=1}^n|e_{ik}||\xi^2_k(s-\tau_2(s))-{\tilde\xi}^2_k(s-\tau_2(s))|ds\\
		&\leq \frac{\displaystyle \max_{1\le i\le d}\displaystyle \sum_{j=1}^d(|a_{ij}|+|b_{ij}|)\exp{(\gamma t)}}{\Gamma(\alpha_i)}
		\int_0^t (t-s)^{\alpha_i-1}\exp{(-\gamma(t- s))}\frac{{(\xi-{\tilde\xi})^1}^*(s)}{\exp{(\gamma s)}}ds\\ 
		&\hspace{1.5cm}+\frac{\displaystyle \max_{1\le i\le d}\displaystyle \sum_{k=1}^n|e_{ik}|\exp{(\gamma t)}}{\Gamma(\alpha_i)}
		\int_0^t (t-s)^{\alpha_i-1}\exp{(-\gamma(t- s))}\frac{{(\xi -{\tilde\xi})^2}^*(s)}{\exp{(\gamma s)}}ds \\
		&\leq \frac{\displaystyle \max_{1\le i\le d}\displaystyle \sum_{j=1}^d(|a_{ij}|+|b_{ij}|)\exp{(\gamma t)}}{\Gamma(\alpha_i)\gamma^{\alpha_i}}
		\int_0^t (t-s)^{\alpha_i-1}\exp{(-\gamma(t- s))}ds \|\xi^1-\tilde{\xi}^1\|_{\gamma} \\ 
		&\hspace{1.5cm}+\frac{\displaystyle \max_{1\le i\le d}\displaystyle \sum_{k=1}^n|e_{ik}|\exp{(\gamma t)}}{\Gamma(\alpha_i)\gamma^{\alpha_i}}
		\int_0^t (t-s)^{\alpha_i-1}\exp{(-\gamma(t- s))}ds\|\xi^2-{\tilde\xi}^2\|_{\gamma}
	\end{align*}
\begin{align*}
&\leq \frac{\displaystyle \max_{1\le i\le d}\displaystyle \sum_{j=1}^d(|a_{ij}|+|b_{ij}|)\exp{(\gamma t)}}{\gamma^{\alpha_i}}
\|\xi^1-\tilde{\xi}^1\|_{\gamma}
+\frac{\displaystyle \max_{1\le i\le d}\displaystyle \sum_{k=1}^n|e_{ik}|\exp{(\gamma t)}}{\gamma^{\alpha_i}}
\|\xi^2-\tilde{\xi}^2\|_{\gamma},
\end{align*}
where we have used the fact that $$\int_0^\infty u^{\alpha_i-1}\exp{(-u)}du=\Gamma(\alpha_i).$$
It implies that
\begin{align*}
\frac{|(\mathcal{T}_{\Phi_0}^1\xi)_i(t)-(\mathcal{T}_{\Phi_0}^1\xi)_i(t)|}{\exp{(\gamma t)}}
\leq \frac{\displaystyle\max_{1\leq i\leq d}\displaystyle\sum_{j=1}^d|a_{ij}|+|b_{ij}|}{\gamma^{\alpha_i}} \|\xi^1-{\tilde\xi}^1\|_\gamma+\frac{\displaystyle\max_{1\leq i\leq d}\displaystyle\sum_{k=1}^n|e_{ik}|}{\gamma^{\alpha_i}} \|\xi^2-{\tilde\xi}^2\|_{\gamma}
\end{align*}
for all $i=1,2,...,d,\; t\in [0,T]$. Thus,
\begin{align}
\notag \|\mathcal{T}_{\Phi_0}^1\xi
-\mathcal{T}_{\Phi_0}^1\xi\|_\gamma&\leq \max_{1\leq l\leq d}\frac{\displaystyle\max_{1\leq i\leq d}\displaystyle\sum_{j=1}^d|a_{ij}|+|b_{ij}|}{\gamma^{\alpha_l}} \|\xi^1-{\tilde\xi}^1\|_\gamma\\
&\hspace{2cm}+\max_{1\leq l\leq d}\displaystyle\frac{\displaystyle\max_{1\leq i\leq d}\displaystyle\sum_{k=1}^n|e_{ik}|}{\gamma^{\alpha_l}} \|\xi^2-{\tilde\xi}^2\|_{\gamma}.\label{tau1}
\end{align}
Furthermore,
\begin{align*}
&\left|(\mathcal{T}_{\Phi_0}^2\xi)_j(t)-(\mathcal{T}_{\Phi_0}^2\xi)_j(t)\right|\\
&=\left|\sum_{k=1}^dc_{jk}(\xi^1_k(t)-{{\tilde\xi}^1}_k(t))+\sum_{k=1}^nd_{jk}(\xi^2_k(t-\tau_3(t))-{\tilde\xi^2}_k(t-\tau_3(t)))\right|\\
&\le \sum_{k=1}^d\left|c_{jk}(\xi^1_k(t)-{\tilde{\xi}^1}_k(t))\right|+\sum_{k=1}^n\left|d_{jk}(\xi^2_k(t-\tau_3(t))-{\tilde\xi^2}_k(t-\tau_3(t)))\right|\\
&\le\displaystyle \max_{1\le j\le n}\displaystyle \sum_{k=1}^d|c_{jk}|\exp{(\gamma t)}
\|\xi^1-{\tilde\xi}^1\|_{\gamma}
+\displaystyle \max_{1\le j\le n}\displaystyle \sum_{k=1}^n|d_{jk}|\exp{(\gamma t)}
\|\xi^2-{\tilde\xi}^2\|_{\gamma}
\end{align*}
for $1\le j\le n,\;t\in[0,T]$, which leads to that
\begin{align} \label{tau2}
\|\mathcal{T}_{\Phi_0}^2\xi-\mathcal{T}_{\Phi_0}^2\tilde{\xi}\|_\gamma\leq \displaystyle\max_{1\leq j\leq n}\displaystyle\sum_{k=1}^d|c_{jk}| \|\xi^1-{\tilde\xi}^1\|_\gamma+\displaystyle\max_{1\leq j\leq n}\displaystyle\sum_{k=1}^n|d_{jk}|\|\xi^2-{\tilde\xi}^2\|_\gamma.
\end{align}
By combining \eqref{tau1} and \eqref{tau2}, we have the following estimate
\begin{align*}
\|\mathcal{T}_{\Phi_0}\xi-\mathcal{T}_{\Phi_0}\tilde{\xi}\|_\gamma&\le \max_{1\leq l\leq d}\left( \frac{\displaystyle\max_{1\leq i\leq d}\displaystyle\sum_{j=1}^d|a_{ij}|+|b_{ij}|}{\gamma^{\alpha_l}}+\displaystyle\max_{1\leq j\leq n}\displaystyle\sum_{k=1}^d|c_{jk}|\right) \|\xi^1-{\tilde\xi}^1\|_\gamma\\
&\hspace{2cm}+\max_{1\leq l\leq d}\left(\frac{\displaystyle\max_{1\leq i\leq d}\displaystyle\sum_{k=1}^n|e_{ik}|}{\gamma^{\alpha_l}}+\displaystyle\max_{1\leq j\leq n}\displaystyle\sum_{k=1}^n|d_{jk}|\right)\|\xi^2-{\tilde\xi}^2\|_\gamma.
\end{align*}
Choose $\gamma>0$ large enough such that
\begin{align*}
\max_{1\leq l\leq d}\frac{\displaystyle\max_{1\leq i\leq d}\displaystyle\sum_{j=1}^d|a_{ij}|+|b_{ij}|}{\gamma^{\alpha_l}}+\displaystyle\max_{1\leq j\leq n}\displaystyle\sum_{k=1}^d|c_{jk}|&:=p_1<1,\\
\max_{1\leq l\leq d}\frac{\displaystyle\max_{1\leq i\leq d}\displaystyle\sum_{k=1}^n|e_{ik}|}{\gamma^{\alpha_l}}+\displaystyle\max_{1\leq j\leq n}\displaystyle\sum_{k=1}^n|d_{jk}|&:=p_2<1.
\end{align*}
Then,
$$\|\mathcal{T}_{\Phi_0}\xi-\mathcal{T}_{\Phi_0}\tilde{\xi}\|_\gamma\le \max\left\{p_1,p_2\right\}\|\xi-\tilde\xi\|_\gamma,$$
and thus the operator $\mathcal{T}_{\Phi_0}$ is contractive in $X$ with respect to the norm $\|\cdot\|_\gamma$. Banach fixed point theorem ensures that this operator has a unique fixed point which is also the unique solution to initial value problem \eqref{E1}-\eqref{dkd}. The proof is complete.
\end{proof}
\begin{remark}
The arguments in the proof of Theorem \ref{EU} are independent of the length $T$ of the existence interval. So, the theorem is still valuable on the semi-real line $[0,\infty)$ instead of the finite interval $[0, T]$.
\end{remark}
\begin{remark}
It is worth noting that  when the delays are strictly positive (e.g., $\tau_2(t), \tau_3(t)>0$ for all $t>0$) or constants, then it seems that no assumptions on $D$ and $C$ are required. Indeed, to prove the existence and uniqueness of the solution to the system \eqref{E1}-\eqref{dkd} in this case, the arguments as in the proof of \cite[Lemma 1]{Cui} can be applied. However, when the delays vary arbitrarily as in our current paper, that approach doesn't seem to work. To our knowledge, the most efficient strategy is to use Banach's fixed point theorem as shown above, and thus the assumption $\|C\|_\infty, \|D\|_\infty<1$ is essential.
\end{remark}
\section{Positivity of mixed-order coupled systems with time-varying delays}\label{s4}
Consider the system \eqref{E1}-\eqref{dkd}:
\begin{equation*}
\begin{cases}
^{\!C}D^{\hat{\alpha}}_{0+}x(t)&=Ax(t)+Bx(t-\tau_1(t))+Ey(t-\tau_2(t))+f(t),\; t\in (0,T],\\
y(t)&=Cx(t)+Dy(t-\tau_3(t))+g(t),\; t\in [0,T],\\
x(\cdot)&=\psi(\cdot),\quad y(\cdot)=\varphi(\cdot),\;\text{on} \;[-r,0].
\end{cases}
\end{equation*}
Assume that the conditions as in the section 3 hold.
\begin{definition}
The system \eqref{E1} is positive if for any $f,g\succeq 0$ on $[0,T]$ and the initial conditions $\psi,\varphi\succeq 0$ on $[-r,0]$ subjected the condition (K), its solution $\Phi(\cdot;\psi,\varphi)$ satisfies 
\[\Phi(t;\psi,\varphi)\succeq 0,\;\forall t\in [0,T].\]
\end{definition}
Our task in this section is to find criteria to characterize the positivity of the system \eqref{E1}.
\begin{theorem} \label{D2}
Consider the system \eqref{E1}. Suppose that $A$ is Metzler, $B,E,C,D$ are nonnegative. Moreover, $C$ and $D$ satisfy the conditions \eqref{c1} and \eqref{d1}, respectively.
Then, the system \eqref{E1} is positive.
\end{theorem}
\begin{proof} By Theorem \ref{EU}, the system \eqref{E1} with the initial condition $x(\cdot)=\psi(\cdot),\;y(\cdot)=\varphi(\cdot)$ on $[-r,0]$ has a unique solution on $[0,T].$ Due to the assumption that $A$ is Metzler, there is a positive constant $\rho>0$ such that $A=-\rho I_d+(A+\rho I_d)$ where $A+\rho I_d$ is nonnegative. We rewrite the first equation in the system \eqref{E1} in the form
$$^{\!C}D^{\hat{\alpha}}_{0+}x(t)=-\rho I_dx(t)+(A+\rho I_d)x(t)+Bx(t-\tau_1(t))+Ey(t-\tau_2(t))+f(t),\;t\in (0,T].$$
Hence, by using the variation-of-constants formula for fractional differential equations in \cite[Lemma 3.1]{Cong_Tuan_17}, we have
\begin{equation*}
\begin{cases}
x_i(t)&=E_{\alpha_i}(-\rho t^{\alpha_i})\psi_i(0)\\
&+\int_0^t (t-s)^{\alpha_i-1}E_{\alpha_i,\alpha_i}(-\rho (t-s)^{\alpha_i}) \Big[\displaystyle \sum_{k=1}^n e_{ik}y_k(s-\tau_2(s))+f_i(s)\Big]ds \\
&+\int_0^t (t-s)^{\alpha_i-1}E_{\alpha_i,\alpha_i}(-\rho (t-s)^{\alpha_i}) \Big[\displaystyle \sum_{j=1}^d(a_{ij}+\rho\delta_{ij})x_j(s)+\displaystyle \sum_{j=1}^db_{ij}x_j(s-\tau_1(s))\Big]ds,\\
y_j(t)&=\displaystyle \sum_{i=1}^dc_{ji}x_i(t)+\displaystyle \sum_{k=1}^nd_{jk}y_k(t-\tau_3(t))+g_j(t),
\end{cases}
\end{equation*}
for $i={1,\dots,d}$, $j=1,\dots,n$, and $t\in [0,T]$, where
\begin{equation*}
\delta_{ij}:=\begin{cases}
1,\;\text{if}\; i=j,\\
0,\;\text{if}\;i\neq j.
\end{cases}
\end{equation*}
We first show that if $\psi\succ0$, $\varphi\succeq 0$ on $[-r,0]$ and $f\succeq 0$, $g\succ0$ on $[0,T]$, then $\Phi(t;\psi,\varphi)\succ 0,\;\forall t\in[0,T]$.
By the compatibility condition $(K)$, then $\Phi(0)\succ0.$ Assume, for the sake of contradiction, that $\Phi(t;\psi,\varphi)\nsucc 0,\;t\in[0,T]$. This implies that there is an index $i=1,\dots,d,d+1,\dots, d+n$ and $t>0$ such that
$$\Phi_{i}(t;\psi,\varphi)=0,\quad \Phi_j(t;\psi,\varphi)\geq 0,\; 1\le j \le d+n,\;j\ne i.$$
Put
$$t_*=\inf\left\{t\in(0,T]:\Phi_{i}(t;\psi,\varphi)=0\;\text{for some}\; 1\leq i\leq d+n\right\}.$$
Then, $t_*>0$ and there is an index $i_0=1,\dots, d+n$ such that $\Phi_{i_0}(t_*;\psi,\varphi)=0$, $\Phi_{i_0}(t;\psi,\varphi)>0$ for $t\in [0,t_*)$,
\begin{equation*}
\Phi_{i}(t;\psi,\varphi)\geq 0, \quad \forall i=1,2,…,d+n,\;i\neq i_0\;\text{and}\;t\in (0,t_*].
\end{equation*}

If $i_0\in\left\{1,\dots,d\right\}$, then by the fact that the Mittag-Leffler functions $E_{\alpha_{i_0}}(\cdot), E_{\alpha_{i_0},\alpha_{i_0}}(\cdot)$ are positive on $\R$, we see 
\begin{align*}
0=\Phi_{i_0}(t_*;\psi,\varphi)&=x_{i_0}(t_*;\psi,\varphi)\\
&=E_{\alpha_{i_0}}(-\rho {t_*}^{\alpha_{i_0}})\psi_{i_0}(0) +\int_0^{t_*} (t_*-s)^{\alpha_{i_0}-1}E_{\alpha_{i_0},\alpha_{i_0}}(-\rho (t_*-s)^{\alpha_{i_0}})f_{i_0}(s)ds \\
&\hspace{0.4cm}+\int_0^{t_*} (t_*-s)^{\alpha_{i_0}-1}E_{\alpha_{i_0},\alpha_{i_0}}(-\rho (t_*-s)^{\alpha_{i_0}})\displaystyle \sum_{j=1}^d(a_{{i_0}j}+\rho\delta_{{i_0}j})x_j(s;\psi,\varphi)ds \\
&\hspace{0.4cm}+\int_0^{t_*} (t_*-s)^{\alpha_{i_0}-1}E_{\alpha_{i_0},\alpha_{i_0}}(-\rho (t_*-s)^{\alpha_{i_0}}) \displaystyle \sum_{j=1}^db_{{i_0}j}x_j(s-\tau_1(s);\psi,\varphi)ds \\
&\hspace{0.4cm}+\int_0^{t_*} (t_*-s)^{\alpha_{i_0}-1}E_{\alpha_{i_0},\alpha_{i_0}}(-\rho (t_*-s)^{\alpha_{i_0}}) \displaystyle \sum_{k=1}^ne_{{i_0}k}y_k(s-\tau_2(s);\psi,\varphi)ds \\
&\ge E_{\alpha_{i_0}}(-\rho {t_*}^{\alpha_{i_0}})\psi_{i_0}(0) >0,
\end{align*}
a contradiction.

If $i_0\in\left\{d+1,\dots,d+n\right\}$, then $\Phi_{i_0}(t_*;\psi,\varphi)=y_{\bar i_0}(t_*;\psi,\varphi)=0$ with $\bar i_0=i_0-d$. Moreover, from the second equation of the system \eqref{E1}, we see that
\begin{align*}
0=y_{\bar i_0}(t_*;\psi,\varphi)&=\displaystyle \sum_{j=1}^dc_{\bar i_0j}x_j(t_*;\psi,\varphi)+\displaystyle \sum_{k=1}^nd_{\bar i_0k}y_k(t_*-\tau_3(t_*);\psi,\varphi)+g_{\bar i_0}(t_*)\ge g_{\bar i_0}(t_*)>0,
\end{align*}
a contradiction. Thus, $\Phi(t,\psi,\varphi)\succ 0$ on $[0,T]$.
	
We now consider the general case where $f,g\succeq 0$ on $[0,T]$ and the initial conditions $\psi,\varphi\succeq0$ on $[-r,0]$. Based on the assumption \eqref{d1}, we can find a unique vector $ m=(I_n-D)^{-1}(C {\bf 1}_d  + {\bf 1}_n) \in \mathbb R^n_+$ such that $$ C {\bf 1}_d + Dm + {\bf 1}_n = m. $$
For each $k \in \mathbb N$, we put $\psi^{k}=\psi+\frac{1}{k}{\bf 1}_d$, $\varphi^{k}=\varphi+\frac{1}{k}m$ on $[-r,0]$, $g^{k}(t)=g(t)+\frac{1}{k}{\bf 1}_n$ on $[0,T]$ and study the system
\begin{equation} \label{E2}
\begin{cases}
^{\!C}D^{\hat{\alpha}}_{0+}x(t)&=Ax(t)+Bx(t-\tau_1(t))+Ey(t-\tau_2(t))+f(t)\\
y(t)&=Cx(t)+Dy(t-\tau_3(t))+g^{k}(t)
\end{cases}
\end{equation}
with the initial condition
\begin{equation}\label{dkdk}
\begin{cases}
x(t)&=\psi^{k}(t),\; t\in [-r,0],\\
y(t)&=\varphi^{k}(t),\; t\in [-r,0].
\end{cases}
\end{equation}
The system \eqref{E2}-\eqref{dkdk} has a unique solution $\Phi^{k}(\cdot;\psi^k,\varphi^k)$. On the other hand, $\Phi^{k}(\cdot;\psi^k,\varphi^k)\succ 0$ on $[0,T]$.

Let $k,l\in\N$ and $k<l$, put $$z^{k,l}(t)=\left(\begin{array}{cc}
x^{k,l}(t) \\ y^{k,l}(t)\end{array}\right)=\Phi^{k}(t;\psi^k,\varphi^k)-\Phi^{l}(t;\psi^l,\varphi^l),\ t\in [-r,T].$$
Then,
\begin{equation}\label{teq}
\begin{cases}
^{\!C}D^{\hat{\alpha}}_{0+}x^{k,l}(t)&=Ax^{k,l}(t)+Bx^{k,l}(t-\tau_1(t))+Ey^{k,l}(t-\tau_2(t)),\;t\in [0,T]\\
y^{k,l}(t)&=Cx^{k,l}(t)+Dy^{k,l}(t-\tau_3(t))+\big(\frac{1}{k}-\frac{1}{l}\big){\bf 1}_n,\;t\in [0,T]
\end{cases}
\end{equation}
and
\begin{equation}\label{dkd33}
\begin{cases}
x^{k,l}(t)&=\big(\frac{1}{k}-\frac{1}{l}\big){\bf 1}_d,\;t\in [-r,0],\\
y^{k,l}(t)&=(\frac{1}{k}-\frac{1}{l})m,\;t\in [-r,0].
\end{cases}
\end{equation}
As shown above, the system \eqref{teq} is positive and thus $z^{k,l}(t) \succ 0,\ t\in [-r,T]$, that is, $\Phi^{k}(t;\psi^k,\varphi^k)\succ \Phi^{l}(t;\psi^l,\varphi^l),\ t\in [-r,T]$. This shows that, for each $t\in[-r,T]$, the sequence $\{\Phi^{k}(\cdot;\psi^k,\varphi^k)\}_{k=1}^\infty$ is monotonically decreasing and there exists $$ \Phi^*(t):=\left(\begin{array}{cc}
x^*(t) \\ y^*(t)\end{array}\right):=\displaystyle\lim_{k\to\infty}\Phi^{k}(t;\psi^k,\varphi^k).$$
On the other hand, it is easy to see that the sequence $\{\Phi^{k}(\cdot;\psi^k,\varphi^k)\}_{k=1}^\infty$ is uniformly bounded on $[0,T]$ which together the similar arguments as in the proof of \cite[Lemma 2.3.2, p. 26]{Lakshmikantham} implies that this sequence is equicontinuous on $[0,T]$. Hence, the its limit $\Phi^*(\cdot)$ is continuous on $[0,T]$.
By Dini's theorem, $\left\{\Phi^{k}(\cdot;\psi^k,\varphi^k)\right\}_{k=1}^\infty$ converges uniformly to $\Phi^*(\cdot)$ and $\Phi^*(\cdot)$ is nonnegative on $[-r,T]$. Noting that on $[-r,0]$
$$\Phi^*(t)=\left(\begin{array}{cc}
\psi(t) \\ \varphi(t)\end{array}\right).$$
Furthermore, from \eqref{E2}, for each $i=1,\dots,d$, $j=1,\dots,n$ and $t\in [0,T]$, we have
\begin{align*}
x_i^{(k)}&=E_{\alpha_i}(-\rho {t}^{\alpha_i})\psi_i^{k}(0) +\int_0^{t} (t-s)^{\alpha_i-1}E_{\alpha_{i},\alpha_i}(-\rho (t-s)^{\alpha_i})f_i(s)ds \\
&\hspace{2cm}+\int_0^{t} (t-s)^{\alpha_i-1}E_{\alpha_i,\alpha_i}(-\rho (t-s)^{\alpha_i})\displaystyle \sum_{j=1}^d(a_{ij}+\rho\delta_{ij})x_j^{(k)}(s)ds \\
&\hspace{2cm}+\int_0^{t} (t-s)^{\alpha_i-1}E_{\alpha_i,\alpha_i}(-\rho (t-s)^{\alpha_i}) \displaystyle \sum_{j=1}^db_{ij}x_j^{(k)}(s-\tau_1(s))ds\\
&\hspace{2cm}+\int_0^{t} (t-s)^{\alpha_i-1}E_{\alpha_i,\alpha_i}(-\rho (t-s)^{\alpha_i}) \displaystyle \sum_{j=1}^ne_{ij}y_j^{(k)}(s-\tau_2(s))ds,\\
y_j^{(k)}(t)&=\displaystyle \sum_{i=1}^dc_{ji}x_i^{(k)}(t)+\displaystyle \sum_{h=1}^nd_{jh}y_h^{(k)}(t-\tau_3(t))+g_j^{(k)}(t).
\end{align*}
Let $k\to \infty$, then
\begin{align*}
x_i^*(t)&=E_{\alpha_i}(-\rho {t}^{\alpha_i})\psi_i(0) +\int_0^{t} (t-s)^{\alpha_i-1}E_{\alpha_{i},\alpha_i}(-\rho (t-s)^{\alpha_i})f_i(s)ds \\
&\hspace{2cm}+\int_0^{t} (t-s)^{\alpha_i-1}E_{\alpha_i,\alpha_i}(-\rho (t-s)^{\alpha_i})\displaystyle \sum_{j=1}^d(a_{ij}+\rho\delta_{ij})x^*_j(s)ds \\
&\hspace{2cm}+\int_0^{t} (t-s)^{\alpha_i-1}E_{\alpha_i,\alpha_i}(-\rho (t-s)^{\alpha_i}) \displaystyle \sum_{j=1}^db_{ij}x^*_j(s-\tau_1(s))ds\\
&\hspace{2cm}+\int_0^{t} (t-s)^{\alpha_i-1}E_{\alpha_i,\alpha_i}(-\rho (t-s)^{\alpha_i}) \displaystyle \sum_{j=1}^ne_{ij}y^*_j(s-\tau_2(s))ds,\\
y^*_{j}(t)&=\displaystyle \sum_{k=1}^dc_{jk}{x^*}_k(t)+\displaystyle \sum_{k=1}^nd_{jk}y^*_k(t-\tau_3(t))+g_j(t)
\end{align*}
for $1\leq i\leq d$, $1\leq j\leq n$ and $t\in [0,T]$.
Due to uniqueness of the solution to the initial value problem \eqref{E1}-\eqref{dkd}, the function $\Phi^*$ is also its solution, so the positivity of this system is verified.	
\end{proof}
The following theorem gives a necessary criterion to ensure the positivity of the system.
\begin{theorem} \label{D1}
Consider the system \eqref{E1} with the conditions \eqref{c1}-\eqref{d1} are satisfied.
Suppose that the system \eqref{E1} is positive for some delays $\tau_k:[0,T]\rightarrow \R_{\geq 0}$ ($1\leq k\leq 3$) satisfying (T1) and
\begin{itemize}
\item[(T2)] $\tau_k(0)>0$ for $k=1,2,3$.
\end{itemize}
Then, $A$ is Metzler, $B,C,D,E$ are nonnegative. As a consequence, the system \eqref{E1} is also positive for arbitrary delays $\tau_k:[0,T]\rightarrow \R_{\geq 0}$ ($1\leq k\leq 3$) provided that (T1) is satisfied.
\end{theorem}
\begin{proof} According to Theorem \ref{EU}, for any $\psi,\varphi\in C([-r,0];\R)$ provided that the compatibility condition (K) is verified, the initial value problem \eqref{E1}-\eqref{dkd} has a unique solution $$\Phi(t;\psi,\varphi)=\left(\begin{array}{cc}
x(t) \\ y(t)\end{array}\right) =\left(\begin{array}{cc}
x_1(t) \\ \vdots \\ x_d(t) \\y_1(t) \\ \vdots \\ y_n(t)\end{array}\right),$$
where
\begin{align*}\begin{cases}
x_i(t)&=\psi_i(0) +\frac{1}{\Gamma(\alpha_i)}\int_0^t (t-s)^{\alpha_i-1} \left[\displaystyle \sum_{j=1}^da_{ij}x_j(s)+\displaystyle \sum_{j=1}^db_{ij}x_j(s-\tau_1(s))\right]ds\\
&\hspace{1.35cm}+\frac{1}{\Gamma(\alpha_i)}\int_0^t (t-s)^{\alpha_i-1}\left[\displaystyle \sum_{k=1}^ne_{ik}y_k(s-\tau_2(s))+f_i(s)\right]ds,\\
y_j(t)&=\displaystyle \sum_{i=1}^dc_{ji}x_i(t)+\displaystyle \sum_{k=1}^nd_{jk}y_k(t-\tau_3(t))+g_j(t)
\end{cases}
\end{align*}
with $i=1,\dots,d$, $j=1,\dots,n$ and $t\in [0,T]$.

Let the system \eqref{E1}-\eqref{dkd} be positive and put $\tau:=\displaystyle\min\left\{\frac{\tau_1(0)}{2},\frac{\tau_2(0)}{2},\frac{\tau_3(0)}{2}\right\}>0.$ We first show that $C$ is a nonnegative matrix. On the contrary, suppose that $C$ has an entry $c_{j_0h_0}<0$. Let $\psi(0)=e_{h_0},\ \varphi(s)=0,\ \forall s \in [-r,-\tau]$, here $e_{h_0}=(0,\dots,1,\dots,0)^{\rm T}$ denotes the unit vector in $\R^d$ with the $h_0$-coordinate equals to 1.
	
Because the continuity of  $x_i(\cdot),\ i=1,\dots,d$, $y_j(\cdot),\ j=1,\dots,n$ and $\tau_3(\cdot)$, we can find $t_0>0$ (small enough) so that on $[0,t_0]$: 
\begin{itemize}
	\item $t-\tau_3(t)\le -\tau,$
	\item $\displaystyle \sum_{i=1,i\ne h_0}^dc_{j_0i}x_i(t)<\frac{|c_{j_0h_0}|}{2},$
	\item $x_{h_0}(t)>\frac{1}{2}$.
\end{itemize} 
Choose a function $g\in C\left([0,+\infty),\R^n_{\geq 0}\right)$ satisfies the condition (K) and $g(s)=0,\ \forall s\ge t_0$, then we have
\begin{align*}
	y_{j_0}(t_0)&=\displaystyle \sum_{i=1}^dc_{j_0i}x_i(t_0)+\displaystyle \sum_{k=1}^nd_{j_0k}y_k(t_0-\tau_3(t_0))+g_{j_0}(t_0) \\
	&=c_{j_0h_0}x_{h_0}(t_0)+\displaystyle \sum_{i=1,i\ne h_0}^dc_{j_0i}x_i(t)\\
	&<\frac{1}{2}c_{j_0h_0}+\frac{|c_{j_0h_0}|}{2}=0,
\end{align*} 
a contradiction.

Next, we assume that $D$ is not a nonnegative matrix, that is, there exists an entry $d_{j_0h_0}<0$. Take $\psi=0$ on $[-r,0],\ \varphi(s)=e_{h_0},\ \forall s \in [-r,-\tau]$. Because the continuity of  $x_i(\cdot),\ i=1,\dots,d$, $y_j(\cdot),\ j=1,\dots,n$ and $\tau_3(\cdot)$, we can find $t_0>0$ (small enough) so that on $[0,t_0]$: 
	\begin{itemize}
		\item $t-\tau_3(t)\le -\tau,$
		\item $\displaystyle \sum_{i=1}^dc_{j_0i}x_i(t)<\frac{|d_{j_0h_0}|}{2}.$
\end{itemize} 
For a function $g\in C\left([0,+\infty),\R^n_{\geq 0}\right)$ verifying the condition (K) and $g(s)=0,\; \forall s\ge t_0$, then 
	\begin{align*}
		y_{j_0}(t_0)&=\displaystyle \sum_{i=1}^dc_{j_0i}x_i(t_0)+\displaystyle \sum_{k=1}^nd_{j_0k}y_k(t_0-\tau_3(t_0))+g_{j_0}(t_0) \\
		&<\frac{|d_{j_0h_0}|}{2}+d_{j_0h_0}\\
		&=\frac{1}{2}d_{j_0h_0}<0,
	\end{align*} 
	a contradiction.

We now prove that $A=(a_{ij})_{1\leq i,j\leq d}$ is Metzler. On the contrary, suppose that $A$ has an entry $a_{i_0j_0}<0$, where $i_0\neq j_0$. Let $f=0$ on $[0,T]$, $\varphi=0$ on $[-r,-\tau]$ and
\begin{equation*}
\psi(t)=\begin{cases}
e_{j_0},\quad&\text{if}\;t=0,\\
0,\quad &\text{if}\;t\in [-r,-\tau],
\end{cases}
\end{equation*}
where $e_{j_0}=(0,\dots,1,\dots,0)^{\rm T}$ is the unit vector in $\R^d$ with the $j_0$-coordinate equals to 1. Since the continuity of $x_i(\cdot),\ i=1,2,...,d$, and $\tau_1(\cdot),\ \tau_2(\cdot)$, we can find $t_0>0$ (small enough) so that on $[0,t_0]$:
\begin{equation}\label{a1}
	t-\tau_k(t)\le -\tau,\;k=1,2,
\end{equation}
and
\begin{itemize}
\item $\displaystyle \sum_{j=1,j\ne j_0}^da_{i_0j}x_j(t)<\frac{|a_{i_0j_0}|}{2},$
\item $x_{j_0}(t)>\frac{1}{2}$.
\end{itemize}
Then,
$$x_{i_0}(t_0)=\frac{1}{\Gamma(\alpha_{i_0})}\int_0^{t_0} (t_0-s)^{\alpha_{i_0}-1} \left[a_{i_0j_0}x_{j_0}(s)+\sum_{j=1,j\ne j_0}^da_{{i_0}j}x_j(s)\right]ds<0,$$
a contradiction.

If $B=(b_{ij})_{1\leq i,j\leq d}$ is not a nonnegative matrix, there is an entry $b_{i_0j_0}<0$. Take $f=0$ on $[0,T]$, $\varphi=0$ on $[-r,-\tau]$ and
\begin{equation*}
\psi(t)=\begin{cases}
0,\quad&\text{if}\;t=0,\\
e_{j_0},\quad &\text{if}\;t\in [-r,-\tau].
\end{cases}
\end{equation*}
By the same arguments as above, there is an $t_0>0$ such that the estimate \eqref{a1} is true and $\displaystyle \sum_{j=1}^da_{i_0j}x_j(t)<|b_{i_0j_0}|$ on $[0,t_0]$. It implies
$$x_{i_0}(t_0)=\frac{1}{\Gamma(\alpha_{i_0})}\int_0^{t_0} (t_0-s)^{\alpha_{i_0}-1} \Big[\sum_{j=1}^da_{i_0j}x_j(s)+b_{i_0j_0}\Big]ds<0,$$
which contradicts the positivity of the system. Thus, $B$ is nonnegative.
	
Finally, if $E=(e_{ij})_{1\leq i\leq d,1\leq j\leq n}$ is not nonnegative, we suppose that there exists an entry $e_{i_0j_0}<0$. Choose $f=0$ on $[0,T]$, $\psi=0$ on $[-r,-\tau]$ and
	\begin{equation*}
		\varphi(t)=\begin{cases}
			0,\quad&\text{if}\;t=0,\\
			e_{j_0},\quad &\text{if}\;t\in [-r,-\tau].
		\end{cases}
	\end{equation*}
There exists $t_0>0$ (small enough) such that \eqref{a1} is verified and $\displaystyle \sum_{j=1}^da_{i_0j}x_j(t)<\frac{|e_{i_0j_0}|}{2}$ for all $t\in[0,t_0]$. Thus,
	$$x_{i_0}(t_0)=\frac{1}{\Gamma(\alpha_{i_0})}\int_0^{t_0} (t_0-s)^{\alpha_{i_0}-1} \Big[e_{i_0j_0}+\sum_{j=1}^da_{{i_0}j}x_j(s)\Big]ds<0,$$
	a contradiction. Hence, $E$ is nonnegative. 
	
Thus, if the initial value problem \eqref{E1}-\eqref{dkd} is positive, then $A$ is Metzler, $B,C,D,E$ are nonnegative. Then, by Theorem \ref{D2}, the system \eqref{E1}-\eqref{dkd} is still positive for every delays $\tau_k: [0,T]\rightarrow \R_{\geq 0}$ $(1\le k \le 3)$ provided that the condition (T1) holds. The proof is complete.
\end{proof}
\section{Asymptotic behavior of solutions to mixed-order positive linear coupled systems with unbounded delays}\label{s5}
Let $\hat{\alpha}=(\alpha_1,...,\alpha_d) \in (0,1]\times...\times(0,1] \subset \R^d$, $A,B\in \R^{d\times d}$, $E\in \R^{d\times n}$, $C\in \R^{n\times d}$, $D\in \R^{n\times n}$, $f\in C([0,\infty);\R^d)$, $g\in C([0,\infty);\R^n)$ and $\tau_k:[0,\infty)\rightarrow \R_{\geq 0}$ $(k=1,2,3)$ is continuous and that
\begin{itemize}
\item[(T3)] $t-\tau_k(t)\ge-r,\quad \forall t\in [0,\infty)$,
\item[(T4)] $\displaystyle\lim_{t\to\infty}(t-\tau_k(t))=\infty.$
\end{itemize}
Consider the following system
\begin{equation} \label{E3}
\begin{cases}
^{\!C}D^{\hat{\alpha}}_{0+}x(t)&=Ax(t)+Bx(t-\tau_1(t))+Ey(t-\tau_2(t))+f(t),\; t\in (0,\infty)\\
y(t)&=Cx(t)+Dy(t-\tau_3(t))+g(t),\; t\in [0,\infty)
\end{cases}
\end{equation}
with the initial condition
\begin{equation}\label{dkd3}
\begin{cases}
x(\cdot)&=\psi(\cdot)\\
y(\cdot)&=\varphi(\cdot)
\end{cases}
\end{equation}
on $[-r,0]$, where $\psi:[-r,0]\rightarrow \R^d,\varphi:[-r,0]\rightarrow \R^n$ are continuous and verify the condition (K).
$C\psi(0)+D\varphi(-\tau_3(0))+g(0)=\varphi(0)$.
\begin{definition}
The system \eqref{E3}-\eqref{dkd3} is globally attractive if for any the initial condition $\psi\in C([-r,0];\R^d),\varphi\in C([-r,0];\R^n)$ which satisfy the compatibility condition (K), then $\lim_{t\to\infty}\|\Phi(t;\psi,\varphi)\|=0$. In the case the system \eqref{E3} is positive, it is globally attractive if for any the initial condition $\psi\in C([-r,0];\R^d_{\geq 0})$, $\varphi\in C([-r,0];\R^n_{\geq 0})$ subjected the compatibility condition (K), we have $\lim_{t\to\infty}\|\Phi(t;\psi,\varphi)\|=0$.
\end{definition}
Our main contribution in this paper is to propose a criterion on the global attractivity of mixed-order positive linear coupled systems with unbounded delays as below.
\begin{theorem} \label{Asy1}
Suppose that $A$ is Metzler, $B,C,D,E$ are nonnegative and the conditions \eqref{c1}-\eqref{d1} are satisfied.
Moreover, we assume that $f\in C([0,\infty);\R_{\geq 0}^d)$, $g\in C([0,\infty);\R_{\geq 0}^n)$ such that $\displaystyle\lim_{t\to\infty}f(t)=0$, $\displaystyle\lim_{t\to\infty}g(t)=0$. Then, the system \eqref{E3}-\eqref{dkd3} is globally attractive if
$A+B+E(I_n-D)^{-1}C$ is Hurwitz.
\end{theorem}
The essential idea in the proof of Theorem \ref{Asy1} is applying a specific characteristic of fractional order differentiable functions obtained in the works of G. Vainikko \cite{Vainikko_16}, and A. Kubica and K. Ryszewska \cite{Kubica} together with skillful use of arguments comparing the solutions to fractional-order positive systems.
Before going into the detailed proof of the theorem, we need the following preparatory results.

Consider the following system
\begin{equation} \label{M1}
\begin{cases}
^{\!C}D^{\hat{\alpha}}_{0+}x(t)&=Ax(t)+Bx(t-\tau_1(t))+Ey(t-\tau_2(t))+f(t),\; t\in (0,\infty),\\
y(t)&=Cx(t)+Dy(t-\tau_3(t)),\; t\in [0,\infty),\\
x(t)&=0,\; y(t)=0,\;t\in [-r,0].
\end{cases}
\end{equation}
\begin{proposition}\label{p1}
Suppose that $A$ is Metzler, $B,C,D,E$ are nonnegative and the conditions \eqref{c1}-\eqref{d1} are satisfied.
Moreover, we assume that $f\in C([0,\infty);\R_{\geq 0}^d)$ and $\displaystyle\lim_{t\to\infty}f(t)=0$. Then, if $A+B+E(I_n-D)^{-1}C$ is Hurwitz, the solution $\Phi^1(\cdot;0,0)=(x^1(\cdot;0,0),y^1(\cdot;0,0))^{\rm T}$ of \eqref{M1} converges to 0 as $t\to\infty$, that is, $$\displaystyle\lim_{t\to\infty}x^1(t;0,0)=0,\  \displaystyle\lim_{t\to\infty}y^1(t;0,0)=0.$$
\end{proposition}
\begin{proof}
{\bf{Case 1:}} The matrix $C$ satisfies the condition 
\[\sum_{j=1}^dc_{ij}>0,\quad\forall i=\overline{1,n}.\] Due to the fact that $f\in C([0,\infty);\R^d_{\geq 0})$ and $\displaystyle\lim_{t\to\infty}f(t)=0$, we can find a function $\hat{f}(\cdot)\in C([0,\infty);\R^d_{\geq 0})$ such that
\begin{itemize}
\item $\hat{f}(\cdot)$ is decreasing,
\item ${\dot{\hat f}}(\cdot)$ is a piecewise continuous function,
\item $\hat{f}(t)\succeq f(t)$ for all $t\geq 0$,
\item $\displaystyle\lim_{t\to \infty}\hat{f}(t)=0$.
\end{itemize}
Consider the system
\begin{equation} \label{M11}
\begin{cases}
^{\!C}D^{\hat{\alpha}}_{0+}x(t)&=Ax(t)+Bx(t-\tau_1(t))+Ey(t-\tau_2(t))+\hat{f}(t),\; t\in (0,\infty),\\
y(t)&=Cx(t)+Dy(t-\tau_3(t)),\; t\in [0,\infty),\\
x(t)&=0,\quad y(t)=0,\quad t\in [-r,0],
\end{cases}
\end{equation}
and denote its solution by $(\hat{x}^1(\cdot;0,0),\hat{y}^1(\cdot,0,0))^{\rm T}$. Then, by the comparison principle as in \cite[Proposition 2]{GCM_20}, $x^1(\cdot;0,0)\preceq \hat{x}^1(\cdot;0,0)$, $y^1(\cdot;0,0)\preceq \hat{y}^1(\cdot;0,0)$ on $[0,\infty)$. The proof of the proposition is completed by showing $$\lim_{t\to\infty}\hat{x}^1(t;0,0)=0,\;\lim_{t\to\infty}\hat{y}^1(t;0,0)=0.$$
We do this by using the following observations.

{\bf Step 1:} Choose $\lambda^1\in \R^d_{\ge0}$ large enough such that $$(A+B+E(I_n-D)^{-1}C)\lambda^1 +E(I_n-D)^{-1}\epsilon {\bf 1}_n+\hat{f}(0)\prec 0$$ 
and define $\xi^1:=(I_n-D)^{-1}(C\lambda^1+\varepsilon {\bf 1}_n)$ for some positive constant $\varepsilon$. In this step, we will prove that $\hat{x}^1(\cdot;0,0)\preceq \lambda^1$ and $\hat{y}^1(\cdot;0,0)\preceq \xi^1$ on $[0,\infty)$. To do this, we set $x(t)=\lambda^1-\hat{x}^1(t;0,0)$, $y(t)=\xi^1-\hat{y}^1(t;0,0)$, $t\in [-r,\infty)$. Then, $x(t)=\lambda^1$, $y(t)=\xi^1$ on $[-r,0]$, and
\begin{align*}
^{\!C}D^{\hat{\alpha}}_{0+}x(t)&=-^{\!C}D^{\hat{\alpha}}_{0+}\hat{x}^1(t;0,0)\\
&=-\left(A\hat{x}^1(t;0,0)+B\hat{x}^1(t-\tau_1(t);0,0)+E\hat{y}^1(t-\tau_2(t);0,0)+\hat{f}(t)\right)\\
&=Ax(t)+B x(t-\tau_1(t))+E y(t-\tau_2(t))
-\left((A+B)\lambda^1+E\xi^1+\hat{f}(t)\right),\;t>0,\\
y(t)&=\xi^1 -\hat{y}^1(t;0,0)\\
&=\xi^1 -(C\hat{x}^1(t)+D\hat{ y}^1(t-\tau_3(t);0,0))\\
&=Cx(t)+Dy(t-\tau_3(t))+(\xi^1-C\lambda^1-D\xi^1)\\
&=Cx(t)+Dy(t-\tau_3(t))+\varepsilon{\bf 1}_n,\; t\geq 0.
\end{align*}
Notice that $(A+B)\lambda^1+E\xi^1+\hat{f}(t)\preceq 0$, the system above is positive and thus $x(t)\succeq 0$, $y(t)\succeq 0$ for all $t\geq 0$. This implies that $\hat{x}^1(t;0,0)\preceq \lambda^1$, $\hat{y}^1(t;0,0)\preceq \xi^1$ on $[0,\infty)$.

{\bf Step 2:} We will show that there are $t_1>0$ and $q \in (0,1)$ so that
\begin{equation}\label{e1}
\hat{x}^1(t;0,0)\prec q\lambda^1,\quad \hat{y}^1(t;0,0)\prec q\xi^1,\; \forall t\geq t_1.
\end{equation}
Denote by $\overline{x}^1(\cdot),\overline{y}^1(\cdot)$ the unique solution to the following system
\begin{equation} \label{B1}
\begin{cases}
^{\!C}D^{\hat{\alpha}}_{0+}{x}(t)&=A{x}(t)+B\lambda^1+E\xi^1+\hat{f}(t),\quad\forall t> 0\\
{y}(t)&=C{x}(t)+D\xi^1+\varepsilon {\bf 1}_n,\quad\forall t\ge 0
\end{cases}
\end{equation}
with the initial condition ${x}(0)=\lambda^1,{y}(0)=\xi^1.$ By the comparison principle, it is easy to see that
\begin{align}
\label{E6}
0\preceq \hat{x}^1(t;0,0)\preceq\overline{x}^1(t)\preceq \lambda^1,\quad 0\preceq \hat{y}^1(t;0,0)\preceq\overline{y}^1(t)\preceq\xi^1, t\geq 0.
\end{align}
Since $A$ is Metzler, there exists $\rho>0$ so that $\rho I_d+A \succeq 0$. Then,
\[
^{\!C}D^{\hat{\alpha}}_{0+}\overline{x}^1(t)=-\rho I_d\overline{x}^1(t)+(\rho I_d+A)\overline{x}^1(t)+B\lambda^1+E\xi^1+\hat{f}(t),\;t>0,
\]
and this solution thus has the following representation
\begin{align*}
\overline{x}_i^1(t)&=E_{\alpha_i}(-\rho t^{\alpha_i})\lambda_i^1 +\int_0^t (t-s)^{\alpha_i-1}E_{\alpha_i,\alpha_i}(-\rho(t-s)^{\alpha_i})\left(\sum_{k=1}^ne_{ik}\xi_k^1 +\hat{f}_i(s)\right)ds \\
&\hspace{1cm}+\int_0^t (t-s)^{\alpha_i-1}E_{\alpha_i,\alpha_i}(-\rho(t-s)^{\alpha_i})\left( \sum_{j=1}^d
(a_{ij}
+\rho \delta_{ij})\overline{x}_j^1(s)+\sum_{j=1}^db_{ij}\lambda_j^1\right) ds
\end{align*}
with all $t\ge0$, $1\leq i\leq d$.
According to Lemma \ref{L3} and Lemma \ref{L4}, for $t>0, 1\leq i\leq d$, we have
\begin{align}
\dot{\overline{x}}_i^1(t)&= -\rho t^{\alpha_i-1}\lambda_i^1E_{\alpha_i,\alpha_i}(-\rho t^{\alpha_i})+t^{\alpha_i-1}E_{\alpha_i,\alpha_i}(-\rho t^{\alpha_i})\sum_{k=1}^ne_{ik}\xi_k^1\notag\\
&\hspace{2cm}+t^{\alpha_i-1}E_{\alpha_i,\alpha_i}(-\rho t^{\alpha_i})\left( \sum_{j=1}^d
(a_{ij}
+\rho \delta_{ij})\lambda_j^1+\sum_{j=1}^db_{ij}\lambda_j^1+\hat{f}_i(0)\right)\notag\\
&\hspace{2cm}+\int_0^t (t-s)^{\alpha_i-1}E_{\alpha_i,\alpha_i}(-\rho(t-s)^{\alpha_i}) \left(\sum_{j=1}^d(a_{ij}
+\rho \delta_{ij}) \dot{\overline{x}}_j^1(s)+\dot{\hat{f}}_i(s)\right)ds\notag\\
&=t^{\alpha_i-1}E_{\alpha_i,\alpha_i}(-\rho t^{\alpha_i})\left( \sum_{j=1}^da_{ij}
\lambda_j^1+\sum_{j=1}^db_{ij}\lambda_j^1+\sum_{k=1}^ne_{ik}\xi_k^1+\hat{f}_i(0)\right)\notag\\
&\hspace{2cm}+\int_0^t (t-s)^{\alpha_i-1}E_{\alpha_i,\alpha_i}(-\rho(t-s)^{\alpha_i}) \sum_{j=1}^d(a_{ij}
+\rho \delta_{ij}) \dot{\overline{x}}_j^1(s)ds\notag\\
&\hspace{2cm}+\int_0^t (t-s)^{\alpha_i-1}E_{\alpha_i,\alpha_i}(-\rho(t-s)^{\alpha_i})\dot{\hat{f}}_i(s)ds.\label{e2}
\end{align}
Furthermore, based on the arguments as in \cite[Lemma 11]{Kubica} (see also \cite[Theorem 2.1.2]{Brunner}), for $1\leq i\leq d$, it is worth to point out that $\overline{x}^1_i\in C^1(0,T]$ and the limit $\displaystyle\lim_{t\to 0+}t^{1-\alpha_i}\dot{\overline{x}}^1_i(t)$ exists which together \eqref{e2} and the fact
$$\int_0^t (t-s)^{\alpha_i-1}E_{\alpha_i,\alpha_i}(-\rho(t-s)^{\alpha_i})\dot{\hat{f}}_i(s)ds<0,\;t>0,$$
implies that
\begin{equation}\label{e3}
\lim_{t\to 0+}\dot{\overline{x}}^1_i(t)=-\infty.
\end{equation}
We now assume that there exists $t>0$ and $i\in \{1,\cdots,d\}$ such that $\dot{\overline{x}}^1_i(t)=0$. Set $t_0=\inf\{t>0: \dot{\overline{x}}^1_i(t)=0\;\text{for some}\; 1\leq i\leq d\}$. Then, $t_0>0$ and there is $i_0\in \{1,\cdots,d\}$ so that $\dot{\overline{x}}^1_{i_0}(t)<0$ on $(0,t_0)$, $\dot{\overline{x}}^1_{i_0}(t_0)=0$ and $\dot{\overline{x}}^1_{j}(t)\leq 0$ on $(0,t_0]$ for all $1\leq j\leq d$, $j\ne i_0$. However, in this case, we have
\begin{align*}
\dot{\overline{x}}^1_{i_0}(t_0)&=t^{\alpha_{i_0}-1}_0E_{\alpha_{i_0},\alpha_{i_0}}(-\rho t^{\alpha_{i_0}}_0)\left( \sum_{j=1}^da_{i_0j}\lambda^1_j+\sum_{j=1}^db_{i_0j}\lambda^1_j+\sum_{k=1}^ne_{i_0k}\xi^1_k+\hat{f}_{i_0}(0)\right)\\
&\hspace{2cm}+\int_0^{t_0} (t_0-s)^{\alpha_{i_0}-1}E_{\alpha_{i_0},\alpha_{i_0}}(-\rho(t_0-s)^{\alpha_{i_0}}) \sum_{j=1}^d(a_{i_0j}+\rho \delta_{i_0j}) \dot{\overline{x}}^1_j(s)ds\\
&\hspace{2cm}+\int_0^{t_0} (t_0-s)^{\alpha_{i_0}-1}E_{\alpha_{i_0},\alpha_{i_0}}(-\rho(t_0-s)^{\alpha_{i_0}})\dot{\hat{f}}_{i_0}(s)ds<0,
\end{align*}
a contradiction. Thus, $ \dot{\overline{x}}^1(t) \prec 0$ for all $t>0$ which leads to that ${\overline{x}}^1$ is strictly decreasing on $[0,\infty)$. Noting that $0 \preceq \overline{x}^1(t)\preceq \lambda^1,\ \forall t\geqslant0$. Hence, the following limit $ \displaystyle \lim_{t \to\infty}\overline{x}^1(t)=x^*$ exists. Let $\mathcal{L}$ be the Laplace transform. By the final value theorem (see, e.g., \cite[The Final Value Theorem, p. 20]{Graf}), we obtain
\begin{align*}
\lim_{s\to +0}s\mathcal{L}\{^{C}D^{\hat{\alpha}}_{0+}\overline{x}^1(\cdot)\}=\lim_{t\to\infty} {^{C}D^{\hat{\alpha}}_{0+}\overline{x}^1(t)}
=\lim_{t\to\infty}(A\overline{x}^1(t)+B\lambda^1+E\xi^1)
=Ax^*+B\lambda^1+E\xi^1.
\end{align*}
Moreover, based the arguments as in the proof of \cite[Lemma 3.1]{Cong_Tuan_17}, then
\[\mathcal{L}\{^{C}D^{\hat{\alpha}}_{0+}\overline{x}^1_i(\cdot)\}(s)=s^{\alpha_i} \mathcal{L}\{\overline{x}^1_i(\cdot)\}(s)-s^{\alpha_i-1}\lambda^1_i,\;i=1,\dots,d.\]
Hence,
\begin{align*}
\lim_{s\to 0^+}s\mathcal{L}\{^{C}D^{\hat{\alpha}}_{0+}\overline{x}^1(\cdot)\}&=\lim_{s\to 0^+}s[s^{\alpha_1} \mathcal{L}\{\overline{x}^1_1(\cdot)\}(s)-s^{\alpha_1-1}\lambda^1_1,\dots,s^{\alpha_d} \mathcal{L}\{\overline{x}^1_d(\cdot)\}(s)-s^{\alpha_d-1}\lambda^1_d]\\
&=\lim_{s\to 0^+}[s^{\alpha_1}(s \mathcal{L}\{\overline{x}^1_1(\cdot)\}(s)-\lambda^1_1),
\dots,s^{\alpha_d}(s \mathcal{L}\{\overline{x}^1_d(\cdot)\}(s)-\lambda^1_d)]\\
&=0
\end{align*}
thanks to
\[
\lim_{s\to 0^+}s \mathcal{L}\{\overline{x}^1_j(\cdot)\}(s)=\lim_{t\to \infty}\overline{x}^1_j(t)=x^*_j,\ 1\leq j\leq d.
\]
This leads to
\begin{align*}
\lim_{t\to\infty}\overline{x}^1(t)&=-A^{-1}(B\lambda^1+E\xi^1):={\hat\lambda}^1\prec \lambda^1,\\ 
\lim_{t\to\infty}\overline{y}^1(t)&=C{\hat\lambda}^1+D\xi^1+\varepsilon 1_n\prec \xi^1. 
\end{align*}
Fix $q\in (0,1)$ be the constant with ${\hat\lambda}^1\prec q\lambda^1$, $C{\hat\lambda}^1+D\xi^1+\varepsilon {\bf 1}_n\prec q\xi^1$ which together the fact that $\overline{x}^1(\cdot)$ is strictly decreasing on $[0,\infty)$ shows that there is $t_1>0$ so that
$$\hat{x}^1(t;0,0)\preceq \overline{x}^1(t)\prec q\lambda^1,\quad \hat{y}^1(t;0,0)\preceq \overline{y}^1(t)\prec q\xi^1, \quad\forall t\geq t_1.$$

{\bf Step 3:} In this step, we will establish a increasing sequence $\{T_m\}_{m=1}^\infty$ with $ \displaystyle \lim_{m\to\infty}T_m=\infty$ which satisfies 
	\begin{align} \label{B2-2}
		\hat{x}^1(t;0,0)\prec q^m\lambda^1,\quad \hat{y}^1(t;0,0)\prec q^m\xi^1,\quad \forall t\in [T_m,T_{m+1}].
	\end{align}
Since $\displaystyle \lim_{t\to \infty}(t-\tau_k(t))=\infty,k=1,2,3$, there is $\hat{t}_1>t_1$ such that $t-\tau_k(t)>t_1$ for all $t>\hat{t}_1.$ Let $T_1=\hat{t}_1>t_1$, then
	\begin{align*}
		&\hat{x}^1(t;0,0)\preceq \lambda^1,\quad \hat{y}^1(t;0,0)\preceq \xi^1,\quad \forall t\in [0,T_{1}],\\
		&\hat{x}^1(t;0,0)\prec q\lambda^1,\quad \hat{y}^1(t;0,0)\prec q\xi^1,\quad \forall t\ge T_1.
	\end{align*}
Set $\tilde{x}^1(t)=\hat{x}^1(t+T_1;0,0),\tilde{y}^1(t)=\hat{y}^1(t+T_1;0,0),\; t\geq0$. Due to the fact that $$\hat{x}^1(t+T_1;0,0)\prec  q\lambda^1,\   \hat{y}^1(t+T_1;0,0) \prec q\xi^1$$ for all $t \geq 0$,  for each $i=1,\dots,d$, we see that 
		\begin{align*}
		^{\!C}D^{\alpha_i}_{0+}\tilde{x}^1_i(t)&=(^{\!C}D^{\alpha_i}_{0+}\hat{x}^1_{i}(. +T_1;0,0))(t) =\frac {1}{\Gamma(1-\alpha_i)}\int_{0}^{t
		} (t-s)^{-\alpha_i} \dot{\hat{x}}^1_{i}(s+T_1;0,0)ds\\
		&=\frac {1}{\Gamma(1-\alpha_i)}\int_{T_1}^{t+T_1
		} (t+T_1-u)^{-\alpha_i} \dot {\hat{x}}^1_{i}(u;0,0)du\\
		&=\frac {1}{\Gamma(1-\alpha_i)}\int_{0}^{t+T_1
		} (t+T_1-u)^{-\alpha_i} \dot {\hat{x}}^1_{i}(u;0,0)du\\
		&\hspace{1.5cm} - \frac {1}{\Gamma(1-\alpha_i)}\int_{0}^{T_1
		} (t+T_1-u)^{-\alpha_i} \dot {\hat{x}}^1_{i}(u;0,0)du\\
		&=(^{\!C}D^{\alpha_i}_{0+}\hat{x}^1_{i}(\cdot;0,0))(t+T_1)-\frac {1}{\Gamma(1-\alpha_i)}(t+T_1-u)^{-\alpha_i}\hat{x}^1_{i}(u;0,0)|_0^{T_1} \\
		&\hspace{1.5cm}+ \frac {1}{\Gamma(1-\alpha_i)} \int_{0}^{T_1
		} \alpha_i(t+T_1-u)^{-\alpha_i-1} \hat{x}^1_{i}(u;0,0)du\\
		&=\sum_{j=1}^da_{ij}\hat{x}^1_{j}(t+T_1;0,0)+\sum_{j=1}^db_{ij}\hat{x}^1_{j}(t+T_1-\tau_1(t+T_1);0,0)\\
		&\hspace{1.5cm}+\hat{f}_i(t+T_1)+\sum_{k=1}^ne_{ik}\hat{y}^1_{k}(t+T_1-\tau_2(t+T_1);0,0)\\
		&\hspace{1.5cm}-\frac {1}{\Gamma(1-\alpha_i)}\bigg(\frac {\hat{x}^1_i(T_1;0,0)}{t^{\alpha_i}} -\frac {\hat{x}^1_i(0;0,0)}{(t+T_1)^{\alpha_i}}-\int_{0}^{T_1
		} \frac {\alpha_i \hat{x}^1_{i}(u;0,0)}{(t+T_1-u)^{\alpha_i+1}} du \bigg) \\
		&\le\sum_{j=1}^da_{ij}\tilde{x}^1_j(t)+\sum_{j=1}^db_{ij}q\lambda^1_j+\sum_{k=1}^ne_{ik}q\xi^1_k+\hat{f}_i(t+T_1)\\
		&\hspace{1.5cm}+\frac {1}{\Gamma(1-\alpha_i)}\int_{0}^{T_1
		} \frac {\alpha_i\hat{x}^1_i(u;0,0)}{(t+T_1-u)^{\alpha_i+1}} du \\
		&\le\sum_{j=1}^da_{ij}\tilde{x}^1_j(t)+\sum_{j=1}^db_{ij}q\lambda^1_j+\sum_{k=1}^ne_{ik}q\xi^1_k +\hat{f}_i(t+T_1)\\
		&\hspace{1.5cm}+\frac {1}{\Gamma(1-\alpha_i)}\int_{0}^{T_1
		} \frac {\alpha_i \lambda^1_i}{(t+T_1-u)^{\alpha_i+1}}du\\
		&=\sum_{j=1}^da_{ij}\tilde{x}^1_j(t)+\sum_{j=1}^db_{ij}q\lambda^1_j+\sum_{k=1}^ne_{ik}q\xi^1_k+\hat{f}_i(t+T_1)\\
		&\hspace{1.5cm}+\frac {\lambda_i}{\Gamma(1-\alpha_i)} (t+T_1-u)^{-\alpha_i}|_0^{T_1} \\
		&\le\sum_{j=1}^da_{ij}\tilde{x}^1_j(t)+\sum_{j=1}^db_{ij}q\lambda^1_j+\sum_{k=1}^ne_{ik}q\xi^1_k+\hat{f}_i(t+T_1)+\frac {\lambda^1_i}{\Gamma(1-\alpha_i)t^{\alpha_i}}.
	\end{align*}
Put 
\begin{equation*}
	a_i(t)=\begin{cases}		0\ &\text{if}\ t\in [0,1],\\ \frac{\lambda^1_i}{\Gamma(1-\alpha_i)t^{\alpha_i}}\ &\text{if}\ t> 1,	
\end{cases}
\text{and} \; \tilde{f}_i(t)=\hat{f}_i(t+T_1),\ t\ge 0
\end{equation*}	
with each $i=1,\dots,d,$
which implies that
	\[^{\!C}D^{\alpha_i}_{0+}\tilde{x}^1_i(t)\le\sum_{j=1}^da_{ij}\tilde{x}^1_j(t)+\sum_{j=1}^db_{ij}q\lambda^1_j+\sum_{k=1}^ne_{ik}q\xi^1_k+\tilde{f}_i(t)+a_i(t),\]
	here $i=1,\dots,d$ for all $t\ge 1$.
	Thus, we have
	$$^{\!C}D^{\hat{\alpha}}_{0+}\tilde{x}^1(t)\preceq A\tilde{x}^1(t)+qB\lambda^1+qE\xi^1+\tilde{f}(t)+a(t),\quad \forall t>1,$$
	where $a(t)=\left(a_1(t),\dots,a_d(t)\right)^T,\; \tilde{f}(t)=\left(\tilde{f}_1(t),\dots,\tilde{f}_d(t)\right)^T, \; t>0$.
	On the other hand,
	$$\tilde y^1(t)\preceq C\tilde x^1(t)+qD\xi^1+q\varepsilon {\bf 1}_n,\quad \forall t\geq 0.$$
	Let $\tilde{\tilde x}^1(\cdot),\tilde{\tilde y}^1(\cdot)$ be the solution to the following system 
	\begin{equation}\label{systxnga1}
		\begin{cases}
			^{\!C}D^{\hat{\alpha}}_{0+} x(t)=A x(t)+qB\lambda^1+qE\xi^1+\tilde{f}(t)+a(t),\quad \forall t>0,\\
			y(t)=Cx(t)+qD\xi^1+q\varepsilon {\bf 1}_n,\quad \forall t\ge0,
		\end{cases}
	\end{equation}
with  $\tilde{\tilde x}^1_i(0)=\tilde{\lambda}_i^1>\displaystyle\max_{t\in[0,1]} \frac{\tilde{x}^1_i(t)}{E_{\alpha_i}(-\rho t^\alpha_i)},\; i=1,\dots,d$, such that 
\[A\tilde{\lambda}^1+B\lambda^1+E\xi^1+\hat{f}(0)\prec0,\]
where $\tilde{\lambda}^1=(\tilde{\lambda}_1^1,\dots,\tilde{\lambda}_d^1)^T$, and
$\tilde{\tilde y}^1(0)=C\tilde{\tilde x}^1(0)+qD\xi^1+q\varepsilon {\bf 1}_n.$

For $t\in [0,1]$ and for each $i=1,\dots,d$, we have
\begin{align*}
	\tilde{\tilde x}^1_i(t)&=E_{\alpha_i}(-\rho t^{\alpha_i})\tilde{\tilde x}^1_i(0)+\int_0^t (t-s)^{\alpha_i-1}E_{\alpha_i,\alpha_i}(-\rho(t-s)^{\alpha_i}) \sum_{j=1}^d
	(a_{ij}
	+\rho \delta_{ij})\tilde{\tilde x}^1_j(s)ds\notag\\
	&\hspace{1cm}+\int_0^t (t-s)^{\alpha_i-1}E_{\alpha_i,\alpha_i}(-\rho(t-s)^{\alpha_i})\big(q\sum_{j=1}^db_{ij}\lambda^1_j+q\sum_{k=1}^ne_{ik}\xi^1_k +\tilde{f}_i(s)+a_i(s)\big)ds \\
	&\ge E_{\alpha_i}(-\rho t^{\alpha_i})\tilde{\tilde x}^1_i(0)> \tilde{x}^1_i(t).
\end{align*}
Suppose, ad absurdum, there is $t>1$ and $i\in\{1,\dots,d\}$ such that $\tilde{x}^1_i(t)=\tilde{\tilde x}^1_i(t)$. Taking $t_*=\inf\{t>0: \tilde{x}^1_i(t)=\tilde{\tilde x}^1_i(t)\;\text{for some}\; 1\leq i\leq d\}$. Then, $t_*>1$ and there exists $i^*\in \{1,\dots,d\}$ such that $\tilde{x}^1_{i^*}(t)<\tilde {\tilde x}^1_{i^*}(t)$ on $[0,t^*)$, $\tilde{x}^1_{i^*}(t^*)=\tilde{\tilde x}^1_{i^*}(t^*)$ and $\tilde{x}^1_{j}(t)\leq \tilde{\tilde x}^1_{j}(t)$ on $[0,t^*]$ for $j\in\{1,\dots,d\}$, $j\ne i^*$.

On the other hand, it is worth noting that $a_{ii}<0$ for all $1\leq i\leq d$ (thanks to the fact that $A$ is Metzler and Hurwitz) and
\begin{align*}
^{\!C}D^{\alpha_{i^*}}_{0+}\tilde{x}^1_{i^*}(t_*) & \leq a_{i^*i^*}\tilde{x}^1_{i^*}(t_*)+\sum_{j=1,j\ne i^*}^da_{i^*j}\tilde{x}^1_j(t_*)+\sum_{j=1}^db_{i^*j}q\lambda^1_j+\sum_{k=1}^ne_{i^*k}q\xi^1_k\\
&\hspace{1.5cm}+\hat{f}_{i^*}(t_*)+\frac {\lambda_{i^*}}{\Gamma(1-\alpha_{i^*})t_*^{\alpha_{i^*}}}\\
& \leq a_{i^*i^*}\tilde{\tilde{x}}^1_{i^*}(t_*)+\sum_{j=1,j\ne i^*}^da_{i^*j} \tilde{\tilde x}^1_j(t_*)+\sum_{j=1}^db_{i^*j}q\lambda^1_j+\sum_{k=1}^ne_{i^*k}q\xi^1_k\\
&\hspace{1.5cm}+\hat{f}_{i^*}(t_*)+\frac {\lambda_{i^*}}{\Gamma(1-\alpha_{i^*})t_*^{\alpha_{i^*}}} \\
&=^{\!C}D^{\alpha_{i^*}}_{0+}\tilde{\tilde{x}}^1_{i^*}(t_*).
\end{align*}
Thus, by using the comparison principle (see, e.g. \cite[Proposition 26]{TuanC_20}), we obtain that $\tilde{x}^1_{i^*}(t^*)<\tilde{\tilde x}^1_{i^*}(t^*)$, a contradiction. Hence $\tilde x^1(t)\preceq \tilde{\tilde x}^1(t),\;t\geq0$.

We now turn to the system \eqref{systxnga1}.  As shown above, in the same manner, we obtain the following representation of the solution to this system 
\begin{align*}
\tilde{\tilde x}^1_i(t)&=E_{\alpha_i}(-\rho t^{\alpha_i})\tilde{\tilde x}^1_i(0)+\int_0^t (t-s)^{\alpha_i-1}E_{\alpha_i,\alpha_i}(-\rho(t-s)^{\alpha_i}) \sum_{j=1}^d
(a_{ij}
+\rho \delta_{ij})\tilde{\tilde x}^1_j(s)ds\notag\\
&\hspace{1cm}+\int_0^t (t-s)^{\alpha_i-1}E_{\alpha_i,\alpha_i}(-\rho (t-s)^{-\alpha_i})\big(q\sum_{j=1}^db_{ij}\lambda^1_j+q\sum_{k=1}^ne_{ik}\xi^1_k +\tilde{f}_i(s)+a_i(s)\big)ds
\end{align*} 
which implies that
\begin{align*}
	\dot{\tilde{\tilde x}}^1_i(t)&=-\rho t^{\alpha_i-1}E_{\alpha_i}(-\rho t^{\alpha_i})\tilde{\tilde x}^1_i(0)+t^{\alpha_i-1}E_{\alpha_i}(-\rho t^{\alpha_i})\sum_{j=1}^d
	(a_{ij}
	+\rho \delta_{ij})\tilde{\tilde x}^1_j(0)\\
	&\hspace{1cm}+\int_0^t (t-s)^{\alpha_i-1}E_{\alpha_i,\alpha_i}(-\rho(t-s)^{\alpha_i}) \sum_{j=1}^d
	(a_{ij}
	+\rho \delta_{ij})\dot{\tilde{\tilde x}}^1_j(s)ds\notag\\
	&\hspace{1cm}+ t^{\alpha_i-1}E_{\alpha_i,\alpha_i}(-\rho t^{-\alpha_i})\big(q\sum_{j=1}^db_{ij}\lambda^1_j+q\sum_{k=1}^ne_{ik}\xi^1_k +\tilde{f}_i(0)+a_i(0)\big) \\
	&\hspace{1cm}+\int_0^t (t-s)^{\alpha_i-1}E_{\alpha_i,\alpha_i}(-\rho (t-s)^{-\alpha_i})\big(\dot{\tilde{f}}_i(s)+\dot{a}_i(s)\big)ds\\
	&=t^{\alpha_i-1}E_{\alpha_i,\alpha_i}(-\rho t^{-\alpha_i})\big(\sum_{j=1}^d
	a_{ij}\tilde{\lambda}^1_j
	+q\sum_{j=1}^db_{ij}\lambda^1_j+q\sum_{k=1}^ne_{ik}\xi^1_k +\tilde{f}_i(0)\big) \\
	&\hspace{1cm}+\int_0^t (t-s)^{\alpha_i-1}E_{\alpha_i,\alpha_i}(-\rho(t-s)^{\alpha_i}) \sum_{j=1}^d
	(a_{ij}
	+\rho \delta_{ij})\dot{\tilde{\tilde x}}^1_j(s)ds\notag\\
	&\hspace{1cm}+\int_0^t (t-s)^{\alpha_i-1}E_{\alpha_i,\alpha_i}(-\rho (t-s)^{-\alpha_i})\big(\dot{\tilde{f}}_i(s)+\dot{a}_i(s)\big)ds.
\end{align*}
Since$$\sum_{j=1}^d
a_{ij}\tilde{\lambda}^1_j
+q\sum_{j=1}^db_{ij}\lambda^1_j+q\sum_{k=1}^ne_{ik}\xi^1_k +\tilde{f}_i(0)\le \sum_{j=1}^d
a_{ij}\tilde{\lambda}^1_j
+\sum_{j=1}^db_{ij}\lambda^1_j+\sum_{k=1}^ne_{ik}\xi^1_k +\hat{f}_i(0)<0,$$ it follows by the same arguments as in the step 2 that, for $1\leq i \leq d$, the function $\tilde {\tilde x}^1_i(\cdot)$ is also strictly decreasing on $[0,\infty)$. Thus, the limit $\displaystyle\lim_{t\to\infty}\tilde{\tilde x}^1(t)$ exists and it is easy to check that
$$\displaystyle\lim_{t\to\infty}\tilde{\tilde x}^1(t)=-qA^{-1}(B\lambda^1+E\xi^1)=q{\hat\lambda}^1 \prec q^2\lambda^1,$$
which leads to  $$\displaystyle\lim_{t\to\infty}\tilde{\tilde y}^1(t)=-qCA^{-1}(B\lambda^1+E\xi^1)+qD\xi^1+q\varepsilon {\bf 1}_n=q(C{\hat\lambda}_1+D\xi^1+\varepsilon {\bf 1}_n)\prec q^2\xi^1.$$
Thus, there is $t_2>0$ so that
$$\tilde x^1(t)\preceq \tilde{\tilde x}^1(t)\prec q^2 \lambda^1,\tilde y^1(t)\preceq \tilde{\tilde y}^1(t)\prec q^2 \xi^1,\quad \forall t\geq t_2.$$
Let $\hat{t}_2=T_1+t_2$. Due to $\displaystyle \lim_{t\to \infty}(t-\tau_k(t))=\infty,k=1,2,3$, we can find $T_2>\hat{t}_2$ which satisfying $$t-\tau_k(t)>\hat{t}_2,\quad\forall t>T_2, 1\leq k\leq 3.$$
Thus
\begin{align*}
&\hat{x}^1(t;0,0)\prec q \lambda^1,\quad \hat{y}^1(t;0,0)\prec q \xi^1,\quad \forall t\in [T_1,T_2], \\
&\hat{x}^1(t;0,0)\prec q^2 \lambda^1,\quad \hat{y}^1(t;0,0)\prec q^2 \xi^1,\quad \forall t \geq T_2.
\end{align*}
That is, the statement \eqref{B2-2} is true for $m=1$. By applying the same procedure, it is not difficult to check that \eqref{B2-2} is true for $m$ is arbitrary, and thus this completes the proof of this step.

\textbf{{Step 4:}} From \eqref{B2-2}, we have $\displaystyle\lim_{t\to\infty}\hat{x}^1(t;0,0)=0$, $\displaystyle\lim_{t\to\infty}\hat{y}^1(t;0,0)=0$, which completes the proof of this proposition.

\textbf{Case 2:} There exists an index $k_0\in \left\{1,\dots,n\right\}$ such that $c_{k_01}=\dots=c_{k_0d}=0.$ Then, we choose a matrix $\hat{ C}\succeq C$ such that
$0<\displaystyle\sum_{j=1}^d\hat{ c}_{kj}<1,\; k=1,\dots,n$, and $A+B+E(I_n-D)^{-1}\hat{ C}$ is Hurwitz.

Denote by $u(\cdot),\ v(\cdot)$ the unique solution to the following system
\begin{equation} \label{Case2}
	\begin{cases}
		^{\!C}D^{\hat{\alpha}}_{0+}x(t)&=Ax(t)+Bx(t-\tau_1(t))+Ey(t-\tau_2(t))+\hat{f}(t),\; t\in (0,\infty),\\
		y(t)&=\hat{ C}x(t)+Dy(t-\tau_3(t)),\; t\in [0,\infty),\\
		x(t)&=0,\; y(t)=0,\;t\in [-r,0].
	\end{cases}
\end{equation}
As shown above, we have $\displaystyle\lim_{t\to\infty}u(t;0,0)=0$, $\displaystyle\lim_{t\to\infty}v(t;0,0)=0$. By using the positivity of the system, we see
\[0\preceq\hat{x}(t,0,0)\preceq u(t;0,0),\ 0\preceq\hat{y}(t,0,0)\preceq v(t;0,0),\ \forall t\ge 0,\]
which leads to that $\displaystyle\lim_{t\to\infty}\hat{x}^1(t;0,0)=0$, $\displaystyle\lim_{t\to\infty}\hat{y}^1(t;0,0)=0.$
\end{proof}
Let the following system
\begin{equation} \label{M2}
\begin{cases}
	^{\!C}D^{\hat{\alpha}}_{0+}x(t)&=Ax(t)+Bx(t-\tau_1(t))+Ey(t-\tau_2(t)),\; t\in (0,\infty),\\
	y(t)&=Cx(t)+Dy(t-\tau_3(t))+g(t),\; t\in [0,\infty),\\
	x(t)&=\psi(t),\; y(t)=\varphi(t),\;t\in [-r,0],
\end{cases}
\end{equation}
where $g\in C([0,\infty);\R^n)$ and the initial condition satisfies the compatibility condition (K).
\begin{proposition}\label{p2}
Suppose that $A$ is Metzler, $B,C,D,E$ are nonnegative and the conditions \eqref{c1}-\eqref{d1} are satisfied.
	Moreover, we assume that $g\in C([0,\infty);\R^n_{\geq 0})$ and $\displaystyle\lim_{t\to\infty}g(t)=0$. Then, if $A+B+E(I_n-D)^{-1}C$ is Hurwitz, the solution $\Phi^2(\cdot;\psi,\varphi)=(x^2(\cdot;\psi,\varphi),y^2(\cdot;\psi,\varphi))^{\rm T}$ of \eqref{M2} converges to 0 as $t\to\infty$, that is, $$\lim_{t\to\infty}x^2(t;\psi,\varphi)=0,\ \lim_{t\to\infty}y^2(t;\psi,\varphi)=0.$$
\end{proposition}
\begin{proof}
The proof is done through the following four steps.
	
	{\bf Step 1:} Put $G:=\max\{(I_n-D)\displaystyle\sup_{-r\leq t\leq 0}\varphi(t),\displaystyle\sup_{t\geq0}{g}(t)+{\mathbf 1_n}\}$ and choose $\lambda^2\in \R^d_+$ large enough such that $(A+B)\lambda^2+E(I_n-D)^{-1}(C\lambda^2 +G)\prec 0,\ \lambda^2\succ \displaystyle\sup_{-r\leq t\leq 0}\psi(t) $. Choose $\xi^2:=(I_n-D)^{-1}(C\lambda^2+G)$. It is obvious to see that $x^2(\cdot;\psi,\varphi)\preceq \lambda^2$ and $y^2(\cdot;\psi,\varphi)\preceq \xi^2$ on $[0,\infty)$. Indeed, by setting $x(t)=\lambda^2-x^2(t;\psi,\varphi)$, $y(t)=\xi^2-y^2(t;\psi,\varphi)$, $t\in [-r,\infty)$, then $x(t)=\lambda^2-\psi(t)\succeq 0$, $y(t)=\xi^2-\varphi(t)\succeq 0$ on $[-r,0]$, and
	\begin{align}
		^{\!C}D^{\hat{\alpha}}_{0+}x(t)&=-^{\!C}D^{\hat{\alpha}}_{0+}x^2(t;\psi,\varphi)\notag\\
		&=-\left(Ax^2(t;\psi,\varphi)+Bx^2(t-\tau_1(t);\psi,\varphi)+Ey^2(t-\tau_2(t);\psi,\varphi)\right)\notag\\
		&=Ax(t)+B x(t-\tau_1(t))+E y(t-\tau_2(t))-\left((A+B)\lambda^2+E\xi^2\right),\;t>0,\label{t1eq}\\
		y(t)&=\xi^2 -y^2(t-\tau_3(t);\psi,\varphi)\notag\\
		&=\xi^2 -(Cx^2(t;\psi,\varphi)+Dy^2(t-\tau_3(t);\psi,\varphi)+g(t))\notag\\
		&=Cx(t)+Dy(t-\tau_3(t))+(\xi^2-C\lambda^2-D\xi^2-g(t))\notag\\
		&=Cx(t)+Dy(t-\tau_3(t))+(G-g(t)),\;t\geq 0.\label{t2eq}
	\end{align}
By the fact that $(A+B)\lambda^2+E\xi^2\preceq 0$ and $G-g(t) \succeq 0$, the system \eqref{t1eq}-\eqref{t2eq} above is positive and thus $x(t)\succeq 0$, $y(t)\succeq 0$ for all $t\geq 0$. This implies that $$x^2(t;\psi,\varphi)\preceq \lambda^2,\ y^2(t;\psi,\varphi)\preceq \xi^2,\ \forall t\in[0,\infty).$$

	{\bf Step 2:} In this step, we will prove that there is $t_1>0$ and $q \in (0,1)$ so that
\begin{equation}\label{e12}
x^2(t;\psi,\varphi)\prec q\lambda^2,\quad y^2(t;\psi,\varphi)\prec q\xi^2,\; \forall t\geq t_1.
\end{equation}
Denote by $\overline{x}^2(\cdot),\overline{y}^2(\cdot)$ the unique solution to the following system
\begin{equation} \label{B2-3}
\begin{cases}
^{\!C}D^{\hat{\alpha}}_{0+}{x}(t)&=A{x}(t)+B\lambda^2+E\xi^2,\quad\forall t> 0,\\
y(t)&=C{x}(t)+D\xi^2+G,\quad\forall t\ge 0,\\
x(0)&=\lambda^2,\;y(0)=\xi^2.
\end{cases}
\end{equation}
Since
\begin{align}
\label{E62}
0\preceq x^2(t;\psi,\varphi)\preceq\overline{x}^2(t)\preceq \lambda^2,\quad 0\preceq y^2(t;\psi,\varphi)\preceq\overline{y}^2(t)\preceq\xi^2,\quad t\geq 0.
\end{align}
Furthermore, using the same arguments as above, we see that ${\overline{x}}^2(\cdot)$ is strictly decreasing on $[0,\infty)$. Thus,
\begin{align*}
\lim_{t\to\infty}\overline{x}^2(t)&=-A^{-1}(B\lambda^2+E\xi^2):=\hat{\lambda}^2\prec \lambda^2,\\ 
\lim_{t\to\infty}\overline{y}^2(t)&:=\hat{\xi}^2=C\hat{\lambda}^2+D\xi^2+G\prec \xi^2. 
\end{align*}
Let $q\in (0,1)$ be the constant with $\hat{\lambda}^2\prec q\lambda^2$, $\hat{\xi}^2\prec q\xi^2$. Then, we can find a time $t_1>0$ so that
$$x^2(t;\psi,\varphi)\preceq \overline{x}^2(t)\prec q\lambda^2,\quad y^2(t;\psi,\varphi)\preceq \overline{y}^2(t)\prec q\xi^2, \quad\forall t\geq t_1.$$
	
	{\bf Step 3:} In this step, we will point out a decreasing sequence $\{T_m\}_{m=1}^\infty$ with $ \displaystyle \lim_{m\to\infty}T_m=\infty$ which justifies the estimates
\begin{align} \label{B2}
x^2(t;\psi,\varphi)\prec q^m\lambda^2,\quad y^2(t;\psi,\varphi)\prec q^m\xi^2,\quad \forall t\in [T_m,T_{m+1}].
\end{align}
From the assumption that $\displaystyle \lim_{t\to \infty}(t-\tau_k(t))=\infty,k=1,2,3$ and $\displaystyle \lim_{t\to \infty}g(t)=0$, there exists $\hat{t}_1>t_1$ so that $$t-\tau_k(t)>t_1,\ g(t)\preceq qG,\ \forall t>\hat{t}_1.$$ Set $T_1=\hat{t}_1>t_1$, then
\begin{align*}
&x^2(t;\psi,\varphi)\preceq \lambda^2,\quad y^2(t;\psi,\varphi)\preceq \xi^2,\quad \forall t\in [0,T_{1}],\\
&x^2(t;\psi,\varphi)\prec q\lambda^2,\quad y^2(t;\psi,\varphi)\prec q\xi^2,\quad \forall t\ge T_1.
\end{align*}
Define $\tilde{x}^2(t)=x^2(t+T_1;\psi,\varphi),\tilde{y}^2(t)=y^2(t+T_1;\psi,\varphi),\ t\geq0$. Then,
$$^{\!C}D^{\alpha_i}_{0+}\tilde{x}^2(t)\preceq A\tilde{x}^2(t)+qB\lambda^2+qE\xi^2+a(t),\quad \forall t>1,$$
here $a(t)$ is defined as above.
On the other hand,
$$\tilde y^2(t)\preceq C\tilde x^2(t)+qD\xi^2+qG,\quad \forall t>0.$$
Let $\tilde{\tilde x}^2(\cdot),\tilde{\tilde y}^2(\cdot)$ be the solution to the following system
\begin{equation}\label{systxnga}
\begin{cases}
^{\!C}D^{\hat{\alpha}}_{0+} x(t)&=A x(t)+qB\lambda^2+qE\xi^2+a(t),\quad \forall t>0\\
y(t)&=Cx(t)+qD\xi^2+qG,\quad \forall t\ge0
\end{cases}
\end{equation}
with $
\tilde{\tilde x}^2_i(0)=\tilde{\tilde \lambda}_i^2>\displaystyle\max_{t\in[0,1]} \frac{\tilde{x}^2_i(t)}{E_{\alpha_i}(-\rho t^\alpha_i)},\; i=1,\dots,d$ such that
\[A\tilde{\tilde \lambda}^2+B\lambda^2+E\xi^2\prec0,\] where $\tilde{\tilde \lambda}^2=(\tilde{\tilde \lambda}_1^2,\dots,\tilde{\tilde \lambda}_d^2)^T$, and $\tilde{\tilde y}^2(0)=C\tilde{\tilde x}^2(0)+qD\xi^2+qG.$ Then, by using the comparison arguments as above, we see that $\tilde {\tilde x}^2(\cdot)$ is strictly decreasing on $[0,\infty)$. Thus,
$$\displaystyle\lim_{t\to\infty}\tilde{\tilde x}^2(t)=-qA^{-1}(B\lambda^2+E\xi^2)=q\hat{\lambda}^2 \prec q^2\lambda^2,$$
which leads to $$\displaystyle\lim_{t\to\infty}\tilde{\tilde y}^2(t)=-qCA^{-1}(B\lambda^2+E\xi^2)+qD\xi^2+qG=q\hat{\xi}^2\prec q^2\xi^2.$$
Thus, there is $t_2>0$ and the following inequalities are true
$$\tilde x^2(t)\preceq \tilde{\tilde x}^2(t)\prec q^2 \lambda^2,\ \tilde y^2(t)\prec \tilde{\tilde y}^2(t)\preceq q^2 \xi^2,\quad \forall t\geq t_2.$$
Let $\hat{t}_2=T_1+t_2$. Due to $\displaystyle \lim_{t\to \infty}(t-\tau_k(t))=\infty,k=1,2,3$, we can find $T_2>\hat{t}_2$ satisfying $$t-\tau_k(t)>\hat{t}_2,\quad\forall t>T_2, 1\leq k\leq 3.$$
Thus
\begin{align*}
&x^2(t;\psi,\varphi)\prec q \lambda^2,\quad y^2(t;\psi,\varphi)\prec q \xi^2,\quad \forall t\in [T_1,T_2], \\
&x^2(t;\psi,\varphi)\prec q^2 \lambda^2,\quad y^2(t;\psi,\varphi)\prec q^2 \xi^2,\quad \forall t \geq T_2,
\end{align*}
which implies that \eqref{B2} is true for $m=1$. By applying the same procedure as in Proposition \ref{p1}, it is not difficult to check that \eqref{B2} is true for $m$ is arbitrary which completes the proof of this step.

\textbf{{Step 4:}} From \eqref{B2}, we have $\displaystyle\lim_{t\to\infty}x^2(t;\psi,\varphi)=0$, $\displaystyle\lim_{t\to\infty}y^2(t;\psi,\varphi)=0$. The proof is finished.
\end{proof}
\begin{proof}[Proof of Theorem \ref{Asy1}]
The proof is straightforward by applying Proposition \ref{p1} and Proposition \ref{p2}.
\end{proof}
We now focus on the system \eqref{E3} in the homogeneous case, that is, let $f=0$, $g=0$ on $[0,\infty)$ and study the following system:
\begin{equation} \label{Eh3}
\begin{cases}
^{\!C}D^{\hat{\alpha}}_{0+}x(t)&=Ax(t)+Bx(t-\tau_1(t))+Ey(t-\tau_2(t)),\; t\in (0,\infty),\\
y(t)&=Cx(t)+Dy(t-\tau_3(t)),\; t\in [0,\infty),\\
x(t)&=\psi(t),\;y(t)=\varphi(t),\;t\in [-r,0],
\end{cases}
\end{equation}
the initial conditions are continuous and satisfy the compatibility condition
\begin{itemize}
\item[(K1)] $C\psi(0)+D\varphi(-\tau_3(0))=\varphi(0)$.
\end{itemize}

We obtain a necessary and sufficient criterion on the global attractivity of its solutions as follows.

\begin{theorem} \label{Asy2}
Suppose that $A$ is Metzler, $B,C,D,E$ are nonnegative and the conditions \eqref{c1}-\eqref{d1} are satisfied.
Then, the system \eqref{Eh3} is globally attractive if and only if
$A+B+E(I_n-D)^{-1}C$ is Hurwitz.
\end{theorem}
\begin{proof}
\textbf{\textit{ Necessity:}} Suppose that the system \eqref{Eh3} is globally attractive, that is, for any $\psi\in C([-r,0];\R_{\geq 0}^d)$, $\varphi\in C([-r,0];\R_{\geq 0}^n)$ such that the condition (K1) is subjected, we have
\begin{equation}\label{pc1}
\lim_{t\to\infty}\|\Phi(t;\psi,\varphi)\|=0
\end{equation}
with $$\Phi(t;\psi,\varphi):=\left(\begin{array}{cc}
x(t;\psi,\varphi) \\ y(t;\psi,\varphi)\end{array}\right) $$ is the unique solution to the initial value problem \eqref{Eh3}. We will show that $A+B+E(I_n-D)^{-1}C$ is Hurwitz.
Indeed, assume the assertion is false. Due to $\displaystyle\sum_{j=1}^n|d_{ij}|<1,\ 1\le i\le n$, by Lemma \ref{L2}, we see that $(I_n-D)^{-1}\succeq 0$ and thus $E(I_n-D)^{-1}C\succeq 0$. It, together with the assumption that $A$ is Metzler, implies that $A+B+E(I_n-D)^{-1}C$ is also a Metzler matrix. According to Lemma \ref{L1}, then $$\left(A+B+E(I_n-D)^{-1}C\right)\lambda\succeq0,\forall \lambda\succ 0.$$
Choose and fix a positive vector $\lambda\in \R_{+}^d$ and put $\xi:=(I_n-D)^{-1}C\lambda$. We define
$$\widetilde{x}(t):=x(t;\lambda,\xi)-\lambda;\quad \widetilde{y}(t):=y(t;\lambda,\xi)-\xi$$ for all $t\ge -r$, where $$\Phi(t;\lambda,\xi)=\left(\begin{array}{cc}
x(t;\lambda,\xi) \\ y(t;\lambda,\xi)\end{array}\right) $$ is the unique solution to the system \eqref{Eh3} with the initial condition $x(t;\lambda,\xi)=\lambda$ and $ y(t;\lambda,\xi)=\xi$ on $[-r,0]$. Then, it is easy to see that $\widetilde{x}(\cdot),\widetilde{y}(\cdot)$ is the unique solution to the following system
\begin{equation} \label{Eh4}
\begin{cases}
^{\!C}D^{\hat{\alpha}}_{0+}\widetilde{x}(t)&=A\widetilde{x}(t)+B\widetilde{x}(t-\tau_1(t))+E\widetilde{y}(t-\tau_2(t))+\left(A+B+E(I_n-D)^{-1}C\right)\lambda,\;t>0\\
\widetilde{y}(t)&=C\widetilde{x}(t)+D\widetilde{y}(t-\tau_3(t)),\;t>0,\\
\widetilde{x}(t)&=0,\;\widetilde{y}(t)=0,\; t\in [-r,0].
\end{cases}
\end{equation}
On the other hand, since the system \eqref{Eh4} is positive, we observe that $$\widetilde{x}(t)\succeq0,\widetilde{y}(t)\succeq0,\quad\forall t\ge 0.$$ This implies that $x(t;\lambda,\xi)\succeq\lambda$ and $y(t;\lambda,\xi)\succeq \xi$ for all $t\ge 0$, which contradicts \eqref{pc1}. Hence $A+B+E(I_n-D)^{-1}C$ is Hurwitz.
	
\textbf{\textit{Sufficiency:}} The proof is straightforward by applying Theorem \ref{Asy1}.
\end{proof}
\begin{corollary}\label{t1}
Consider the system \eqref{Eh3}
\begin{equation*}
\begin{cases}
^{\!C}D^{\hat{\alpha}}_{0+}x(t)&=Ax(t)+Bx(t-\tau_1(t))+Ey(t-\tau_2(t)),\; t\in (0,\infty)\\
y(t)&=Cx(t)+Dy(t-\tau_3(t)),\; t\in [0,\infty)
\end{cases}
\end{equation*}
with the initial condition
\begin{equation}\label{t2}
\begin{cases}
x(\cdot)&=\lambda\in\R^d,\\
y(\cdot)&=\xi:=(I-D)^{-1}C\lambda
\end{cases}
\end{equation}
on $[-r,0]$. Suppose that the assumptions in Theorem \ref{Asy2} are true. Then, $$\lim_{t\to\infty}\Phi(t;\lambda,\xi)=0,$$
where $\Phi(\cdot;\lambda,\xi)$ is the unique solution to the system \eqref{Eh3}-\eqref{t2}.
\end{corollary}
\begin{proof}
For each $\lambda\in \R^d$, we can find $\lambda^+$, $\lambda^-$ such that $\lambda=\lambda^+ +\lambda^-$, where $\lambda^+,-\lambda^-\in\R^d_{\geq 0}$. Then, we see that $\Phi(\cdot;\lambda,\xi)=\Phi(\cdot;\lambda^+,\xi^{+})+\Phi(\cdot;\lambda^-,\xi^{-})$ on $[-r,\infty)$, here $\xi^+=(I_n-D)^{-1}C\lambda^+\succeq 0$ and $\xi^-=(I_n-D)^{-1}C\lambda^-\preceq 0$. Notice that $\Phi(\cdot;\lambda^-,\xi^{-})=-\Phi(\cdot;-\lambda^-,-\xi^{-})$. Furthermore, by Theorem \ref{Asy2}, $\lim_{t\to\infty}\Phi(t;\lambda^+,\xi^{+})=\lim_{t\to\infty}\Phi(t;-\lambda^-,-\xi^{-})=0$. This leads to that $\lim_{t\to\infty}\Phi(t;\lambda,\xi)=0$.
\end{proof}
As an application of Corollary \ref{t1}, we obtain a result on the smallest asymptotic bound of solutions to  mixed-order positive non-homogeneous linear delay coupled systems below.

Let the system \eqref{E3}
\begin{equation*}
\begin{cases}
^{\!C}D^{\hat{\alpha}}_{0+}x(t)&=Ax(t)+Bx(t-\tau_1(t))+Ey(t-\tau_2(t))+f(t),\; t\in (0,\infty)\\
y(t)&=Cx(t)+Dy(t-\tau_3(t))+g(t),\; t\in [0,\infty)
\end{cases}
\end{equation*}
with the initial condition \eqref{dkd3}
\begin{equation*}
\begin{cases}
x(\cdot)&=\psi(\cdot)\\
y(\cdot)&=\varphi(\cdot)
\end{cases}
\end{equation*}
on $[-r,0]$, where $\psi:[-r,0]\rightarrow \R^d_{\geq 0},\varphi:[-r,0]\rightarrow \R^n_{\geq 0}$ are continuous and
\[
C\psi(0)+D\varphi(-\tau_3(0))+g(0)=\varphi(0).
\]
\begin{theorem} \label{best}
Suppose that $A$ is Metzler, $B,C,D,E$ are nonnegative, $A+B+E(I_n-D)^{-1}C$ is Hurwitz and the following statements hold.
\begin{itemize}
\item[(i)] The conditions \eqref{c1}-\eqref{d1} are satisfied.
\item[(ii)] The matrix $H=(h_{ij})_{1\leq i\leq n,1\leq j\leq d}:=(I_n-D)^{-1}C$ satisfies that for each $i=\overline{1,n}$, there exists $j=\overline{1,d}$ with $h_{ij}>0$.
\end{itemize}
Moreover, we assume that $f\in C([0,\infty);\R_{\geq 0}^d)$, $g\in C([0,\infty);\R_{\geq 0}^n)$, here $\displaystyle{\sup_{t\geq 0}}f(t)=:F\prec \infty$, $\displaystyle\sup_{t\geq 0}g(t)=:G\prec \infty$. Then, 
$$
\sup_{
f\in B_{C([0,\infty);\R^d_{\geq 0})}(0;F)\atop
g\in B_{C([0,\infty);\R^n_{\geq 0})}(0;G)}\limsup_{t\to\infty}\Phi(t;\psi,\varphi)=(x^*,y^*)^{\rm T}
$$
in which $x^*:=-[A+B+E(I_n-D)^{-1}C]^{-1}[F+E(I_n-D)^{-1}G]$, $y^*=(I_n-D)^{-1}(Cx^*+G)$.
\end{theorem}
\begin{proof}
Let $\Phi(\cdot;\psi,\varphi)=(x(\cdot;\psi,\varphi), y(\cdot;\psi,\varphi))^{\rm T}$ be the solution to the system \eqref{E3}-\eqref{dkd3}. Choose $\lambda\in \R^d_+$ such that $\psi(\cdot)\preceq \lambda$ and $\varphi(\cdot)\preceq \xi:=(I_n-D)^{-1}(C\lambda +G)$. Put\begin{equation*}
\begin{cases}
\bar{x}(\cdot)&=\hat{x}(\cdot;\lambda,\xi)-x^*\\
\bar{y}(\cdot)&=\hat{y}(\cdot;\lambda,\xi)-y^*
\end{cases}
\end{equation*}
on $[-r,\infty)$. Here, $(\hat{x}(\cdot;\lambda,\xi), \hat{y}(\cdot;\lambda,\xi))^{\rm T}$ is the unique solution to the following system
\begin{equation}\label{add3}
	\begin{cases}
		^{\!C}D^{\hat{\alpha}}_{0+}x(t)&=Ax(t)+Bx(t-\tau_1(t))+Ey(t-\tau_2(t))+F,\; t\in (0,\infty),\\
		y(t)&=Cx(t)+Dy(t-\tau_3(t))+G,\; t\in [0,\infty),\\
		x(t)&=\lambda,\;y(t)=\xi,\;t\in [-r,0].
	\end{cases}
\end{equation}
It is obvious that $(\bar{x}(\cdot),\bar{y}(\cdot))^{\rm T}$ is the unique solution to the system
\begin{equation}\label{add31}
	\begin{cases}
		^{\!C}D^{\hat{\alpha}}_{0+}x(t)&=Ax(t)+Bx(t-\tau_1(t))+Ey(t-\tau_2(t)),\; t\in (0,\infty),\\
		y(t)&=Cx(t)+Dy(t-\tau_3(t)),\; t\in [0,\infty),\\
		x(t)&=\lambda-x^*,\;y(t)=\xi-y^*,\;t\in [-r,0].
	\end{cases}
\end{equation}
Hence, by Corollary \ref{t1}, we have
\[
\lim_{t\to\infty}\bar{x}(t)=0,\;\lim_{t\to\infty}\bar{y}(t)=0.
\]
It implies that 
\[
\lim_{t\to\infty}\hat{x}(t;\lambda,\xi)=x^*,\;\lim_{t\to\infty}\hat{y}(t;\lambda,\xi)=y^*.
\]
On the other hand, by the comparison principle (see \cite[Proposition 2]{GCM_20}), then on $[-r,\infty)$,
\[
\begin{cases}
		x(t;\psi,\varphi)&\preceq \hat{x}(t;\lambda,\xi),\\
		y(t;\psi,\varphi)&\preceq \hat{x}(t;\lambda,\xi).
	\end{cases}
\]
Thus,
$$
\sup_{f\in B_{C([0,\infty);\R^d_{\geq 0})}(0;F),\atop
g\in B_{C([0,\infty);\R^n_{\geq 0})}(0;G)}\limsup_{t\to\infty}\Phi(t;\psi,\varphi)=(x^*,y^*)^{\rm T}.
$$
The proof is complete.
\end{proof}
\begin{remark}
It is worth noting that if the condition (ii) in Theorem \ref{best} does not occur, then we only have the following weaker estimate
$$
\sup_{f\in B_{C([0,\infty);\R^d_{\geq 0})}(0;F),\atop
	g\in B_{C([0,\infty);\R^n_{\geq 0})}(0;G)}\limsup_{t\to\infty}\Phi(t;\psi,\varphi )\preceq(x^*,y^*)^{\rm T}.
$$	
\end{remark}
\begin{remark}
Our main contributions in this section are an improved version of those already in the literature. In particular, they are generalized, strengthened, and corrected results for \cite[Theorem 2]{ShenLam_16}, \cite[Theorem 1]{TTL_21} and \cite[Theorem 2.2]{Tuan_22}. In addition, these results can also be considered as an extension of \cite[Theorem 1]{Shen_15} and \cite[Theorem 4]{Hieu_18} for the case of fractional-order derivatives.
\end{remark}
Finally, we provide numerical examples to clarify the validity of the proposed theoretical results.
\begin{example} \label{Exam1}
	Consider the delay coupled system
	\begin{equation}\label{ex1}
	\begin{cases}
		^{\!C}D^{\hat{\alpha}}_{0+}x(t)&=Ax(t)+Bx(t-\tau_1(t))+Ey(t-\tau_2(t)),\; t\in (0,\infty),\\
		y(t)&=Cx(t)+Dy(t-\tau_3(t)),\; t\in [0,\infty),\\
	\end{cases}
\end{equation}
where $\hat{\alpha}=(0.45,0.65,0.2)$,
	\begin{align*}
		A=	\begin{pmatrix}
			-2.2 & 0.2 & 0.1 \\
			0.3 & -2.4 & 0.2 \\
			0.5 & 0.2 & -2.3
		\end{pmatrix},\
		B=	\begin{pmatrix}
			0.2 & 0.1 & 0.3 \\
			0.2 & 0.2 & 0.1 \\
			0.1 & 0.3 & 0.5
		\end{pmatrix},\
		E=	\begin{pmatrix}
			0.2 & 0.1 \\
			0.2 & 0.3 \\
			0.3 & 0.4 
		\end{pmatrix},
	\end{align*}
	\begin{align*}
		C=	\begin{pmatrix}
			0.1 & 0.2 & 0.1 \\
			0.1 & 0.3 & 0.1 
		\end{pmatrix},\
		D=	\begin{pmatrix}
			0.1 & 0.2  \\
			0.2 & 0.1 
		\end{pmatrix},
	\end{align*}
	and delays $$\tau_1(t)=1+\frac{1}{2}t\sin^2t,\ \tau_2(t)=\frac{1}{3}+\frac{1}{2}t,\ \tau_3(t)=\frac{1}{2}+\frac{1}{2}t\cos^2t,\;t\geq 0.$$
	It is clear that $A$ is Metzler, $B,C,D,E$ are nonnegative and $A+B+E(I-D)^{-1}C$ is Hurwitz. From Theorem \ref{Asy2}, this implies that the system \eqref{ex1} is globally attractive. In Figure 1, we depict the trajectory of a solution to this system with the initial conditions $\psi=(0.3, 0.2,0.8)^T$, $\varphi=(I-D)^{-1}C\psi\approx(0.2195, 0.2377)^T$ on the interval $[-1,0]$.
	\begin{figure}
		\begin{center}
			\includegraphics[scale=.6]{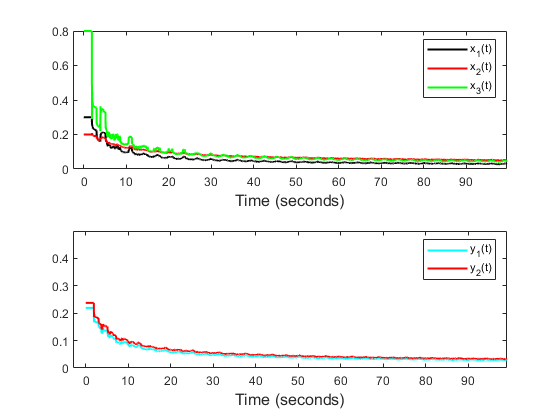}
		\end{center}
		\begin{center}
			\caption{Trajectory of the solution to the system \eqref{ex1} with the initial conditions $\psi=(0.3, 0.2,0.8)^T$, $\varphi\approx(0.2195, 0.2377)^T$ on the interval $[-1,0]$.}
		\end{center}
	\end{figure}
\end{example}
\begin{example}
	Consider the system
	\begin{equation}\label{ex2}
	\begin{cases}
		^{\!C}D^{\hat{\alpha}}_{0+}x(t)&=Ax(t)+Bx(t-\tau_1(t))+Ey(t-\tau_2(t))+f(t),\; t\in (0,\infty),\\
		y(t)&=Cx(t)+Dy(t-\tau_3(t))+g(t),\; t\in [0,\infty),\\
	\end{cases}
\end{equation}
where $\hat{\alpha}=(0.5,0.7,0.3)$, $A,B,C,D,E$, $\tau_1(t),\tau_2(t),\tau_3(t)$ are chosen as in Example \ref{Exam1}, and
	$$f(t)=\big(\frac{2}{1+t}, \frac{t}{e^t}, \frac{\sin^2t}{1+t}\big)^T,\ g(t)=\big(\frac{t}{1+t^2}, \frac{1}{e^t}\big)^T,\;t\geq 0.$$
	By Theorem \ref{Asy1}, the system \eqref{ex2} is globally attractive. In Figure 2 below, a numerical simulation for the trajectories of the solution with initial conditions $\psi=(0.5, 0.1, 1.2)^T$, $\varphi=(I-D)^{-1}[C\psi+g(0)]\approx(0.2740, 0.2831)^T$ on $[-1,0]$.
	\begin{figure}
		\begin{center}
		\includegraphics[scale=.6]{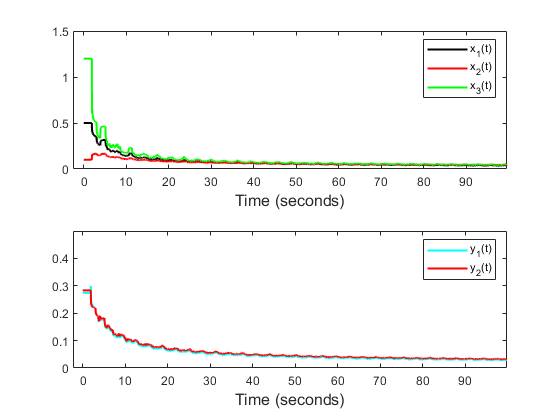}
		\end{center}
		\begin{center}
			\caption{Trajectories of the solution to the system \eqref{ex2} with the initial conditions $\psi=(0.5, 0.1, 1.2)^T$, $\varphi\approx(0.2740, 0.2831)^T$ on $[-1,0]$.}	
		\end{center}
	\end{figure}
\end{example}
\begin{example}
	Consider the system
	\begin{equation}\label{ex3}
	\begin{cases}
		^{\!C}D^{\hat{\alpha}}_{0+}x(t)&=Ax(t)+Bx(t-\tau_1(t))+Ey(t-\tau_2(t))+f(t),\; t\in (0,\infty),\\
		y(t)&=Cx(t)+Dy(t-\tau_3(t))+g(t),\; t\in [0,\infty),\\
	\end{cases}
\end{equation}
where $\hat{\alpha}=(0.55,0.25)$, 
	\begin{align*}
		A=	\begin{pmatrix}
			-2.2 & 0.2  \\
			0.1 & -1.4 
		\end{pmatrix},\
		B=	\begin{pmatrix}
			0.2 & 0.1  \\
			0.2 & 0.3 
		\end{pmatrix},\
		E=	\begin{pmatrix}
			0.2 & 0.5 \\
			0.2 & 0.3 
		\end{pmatrix};
	\end{align*}
	\begin{align*}
		C=	\begin{pmatrix}
			0.4 & 0.2 \\
			0.2 & 0.3
		\end{pmatrix},\
		D=	\begin{pmatrix}
			0.2 & 0.5  \\
			0.3 & 0.1 
		\end{pmatrix};
	\end{align*}
		\begin{align*}
		f(t)=	\begin{pmatrix}
			0.02+0.01\sin t  \\
			\frac{0.1t}{1+t}
		\end{pmatrix},\
		g(t)=	\begin{pmatrix}
			0.1+0.1\cos t  \\
			0.1+0.25 \cdot 2^{\frac{t}{t+1}} 
		\end{pmatrix},
	\end{align*}
	and delays $$\tau_1(t)=1+\frac{1}{2}t\sin^2t,\ \tau_2(t)=\frac{1}{2}+\frac{3}{10}\cos^2t,\ \tau_3(t)=\frac{1}{2}+\frac{1}{2}t\cos^2t\;t\geq 0.$$
It is clear that 
$F:=\sup_{t\geq 0}f(t)=(0.03, 0.1)^T,\ G:=\sup_{t\geq 0}g(t)=(0.2, 0.6)^T.$
According to Theorem \ref{best}, 
of solutions to the system \eqref{ex3} is
	\begin{align*}
		x^*&=-[A+B+E(I-D)^{-1}C]^{-1}(F+E(I-D)^{-1}G)\approx(4.0214,\ 3.6315)^T, \\
		y^*&=(I-D)^{-1}(Cx^*+G)\approx(6.1899,\ 4.8341)^T.
	\end{align*}
In Figure 3, we give a numerical description of the orbits of solutions corresponding to various initial conditions and the smallest asymptotic bound of solutions to this system.
	\begin{figure}
		\begin{center}
		\includegraphics[scale=.6]{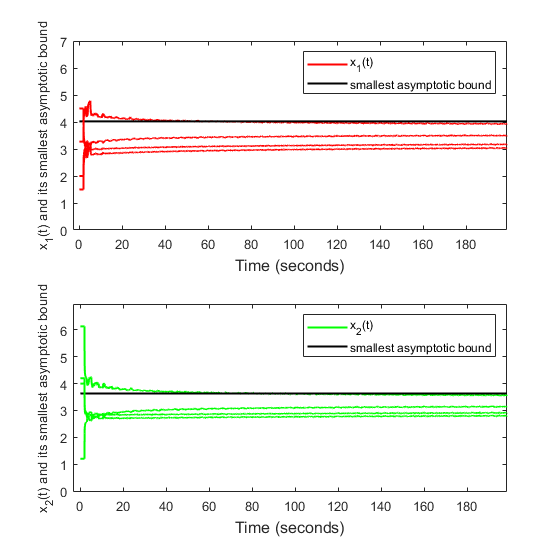}
		\includegraphics[scale=.6]{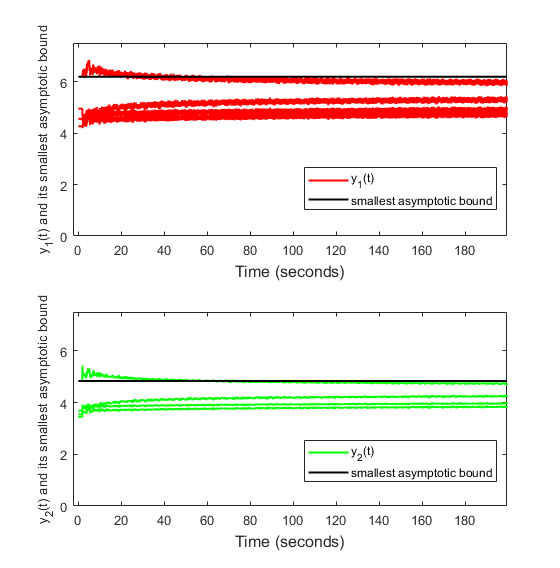}
	\end{center}
	\begin{center}
	\caption{Orbits of various solutions to the system \eqref{ex3} and its smallest asymptotic bound.}
	\end{center}
	\end{figure}
\end{example}
\section*{Acknowledgments}
The authors thank the anonymous referees for useful comments and suggestions which help in improving presentation of the paper. This research is supported by a grant from the Vietnam Academy of Science and Technology under the grant number {\bf CTTH00.03/23-24}. 

\end{document}